\documentclass[10pt]{article} 

\newcommand{\m}[1]{{\bf{#1}}}
\newcommand{\g}[1]{\boldsymbol #1}
\newcommand{\bb}[1]{\mathbb #1}
\newcommand{\tr}{^{\sf T}}

\usepackage{bm}
\usepackage{palatino}
\usepackage{nomencl}
\usepackage{cite}

\title{\bf Modified Legendre-Gauss-Radau Collocation Method for \\ Optimal Control Problems with Nonsmooth Solutions}

\author{Joseph~D.~Eide\thanks{Ph.D.~Student, Department of Mechanical and Aerospace Engineering, University of Florida, Gainesville, Florida 32611-6250. Email: gatoreide@ufl.edu} \\ William W.~Hager\thanks{Distinguished Professor, Department of Mathematics, University of Florida, Gainesville, FL 32611-6250. Fellow, Society for Industrial and Applied Mathematics.  E-mail:  hager@ufl.edu.} \\ Anil V.~Rao\thanks{Professor, Department of Mechanical and Aerospace Engineering, University of Florida, Gainesville, FL 32611-6250. Erich Farber Faculty Fellow and University Term Professor.  Associate Fellow, AIAA.  E-mail:  anilvrao@ufl.edu.  Corresponding Author.} \\ \\ {\em University of Florida} \\ {\em Gainesville, FL 32611}}

\date{}

\usepackage{amsmath,amsfonts,latexsym,enumerate,amsthm,subfig,tikz}
\usepackage{bm}

\hyphenchar\font=-1

\newlength{\widthOfItem}

\renewenvironment{itemize}{%
	\begin{Itemize}
		\setlength{\itemsep}{0.5\baselineskip}
		\setlength{\labelwidth}{2em}
		\setlength{\listparindent}{.32in}%
		\setlength{\leftmargin}{.32in}
		\setlength{\rightmargin}{0in}
		\settowidth{\widthOfItem}{\labelitemi}
		\setlength{\labelsep}{\leftmargin-\widthOfItem}
		
		\singlespacing}{%
	\end{Itemize}}

\newcommand{\mbb}[1]{\mathbb{#1}}

\newcommand{\bmL}{\g{\Lambda}}

\newcommand{\R}{\mbb{R}}

\newcommand{\bmU}{\m{U}}

\newcommand{\bml}{\g{\lambda}}

\newcommand{\rk}{^{(k)}}
\newcommand{\rK}{^{(K)}}
\newcommand{\rktr}{^{(k)^{\sf T}}}

\newcommand{\intoo}{\int^{+1}_{-1}}

\newcommand{\jl}{j = 1,\hdots,N}

\allowdisplaybreaks

\newtheorem{theorem}{Theorem}

\begin{document}

% make the title area
\maketitle

\newcommand{\trs}{^{\sf *T}}

% % THEOREM Environments ---------------------------------------------------
% %These environments are provided as a convenience - feel free to modify if needed

% %Allow page breaks in align environments
% \allowdisplaybreaks
% \usepackage{epsfig} %% for loading postscript figures
% %% The class has several options
% %  onecolumn/twocolumn - format for one or two columns per page
% %  10pt/11pt/12pt - use 10, 11, or 12 point font
% %  oneside/twoside - format for oneside/twosided printing
% %  final/draft - format for final/draft copy
% %  cleanfoot - take out copyright info in footer leave page number
% %  cleanhead - take out the conference banner on the title page
% %  titlepage/notitlepage - put in titlepage or leave out titlepage
% %  
% %% The default is oneside, onecolumn, 10pt, final

% \title{A New Method for Solving Optimal Control Problems with Non-Smooth Optimal Control Profiles}

% %%% first author
% \author{Joseph D Eide\thanks{Ph.D.~Student, Department of Mechanical and Aerospace Engineering, University of Florida, Gainesville, Florida 32611-6250. Email: gatoreide@ufl.edu} \\ Anil V.~Rao\thanks{Associate Professor, Department of Mechanical and Aerospace Engineering, University of Florida, Gainesville, FL 32611-6250. E-mail:  anilvrao@ufl.edu.  Corresponding Author}}

% \bibliographystyle{asmems4}
% \begin{document}
% \newtheorem*{theorem}{Theorem}
% \newtheorem{theorem}{Theorem}

% \maketitle  

\begin{abstract}
  {\noindent}A new method is developed for solving optimal control problems whose solutions are nonsmooth.  The method developed in this paper employs a modified form of the Legendre-Gauss-Radau orthogonal direct collocation method.  This modified Legendre-Gauss-Radau method adds two variables and two constraints at the end of a mesh interval when compared with a previously developed standard Legendre-Gauss-Radau collocation method.  The two additional variables are the time at the interface between two mesh intervals and the control at the end of each mesh interval.  The two additional constraints are a collocation condition for those differential equations that depend upon the control and an inequality constraint on the control at the endpoint of each mesh interval.  The additional constraints modify the search space of the nonlinear programming problem such that an accurate approximation to the location of the nonsmoothness is obtained.  The transformed adjoint system of the modified Legendre-Gauss-Radau method is then developed.  Using this transformed adjoint system, a method is developed to transform the Lagrange multipliers of the nonlinear programming problem to the costate of the optimal control problem.  Furthermore, it is shown that the costate estimate satisfies one of the Weierstrass-Erdmann optimality conditions.  Finally, the method developed in this paper is demonstrated on an example whose solution is nonsmooth.  
\end{abstract}

\renewcommand{\baselinestretch}{2}
\normalsize
\normalfont

\section{Introduction\label{sect:introduction}}

Over the past two decades, direct collocation methods have become increasingly popular for computing the numerical solution of constrained optimal control problems.  A direct collocation method is an implicit simulation method where the state and control are both parameterized and the constraints in the continuous optimal control problem are enforced at a specially chosen set of collocation points.  This approximation of the continuous optimal control problem leads to a finite-dimensional nonlinear programming problem (NLP)\cite{Betts:2010}, and the NLP is solved using well known software \cite{Gill:2005,Biegler:2008}.  Originally, direct collocation methods were developed as $h$ methods (for example, Euler or Runge-Kutta methods) where the time  interval is divided into a mesh and the state is approximated using the same fixed-degree polynomial in each mesh interval.  Convergence in an $h$ method is then achieved by increasing the number of mesh intervals \cite{Betts:2010,Jain:2008,Zhao:2011}.   More recently, a great deal of research as been done in the class of direct {\em Gaussian quadrature orthogonal collocation} methods \cite{Elnagar:1995,Elnagar:1997,Fahroo:2001,Darby:2011a,Darby:2011b,Patterson:2015,Liu:2015,Liu:2018}.  In a Gaussian quadrature collocation method, the state is typically approximated using a Lagrange polynomial where the support points of the Lagrange polynomial are chosen to be points associated with a Gaussian quadrature.   Originally, Gaussian quadrature collocation methods were implemented as $p$ methods using a single interval.   Convergence of the $p$ method was then achieved by increasing the degree of the polynomial approximation.  For problems whose solutions are smooth and well-behaved, a Gaussian quadrature collocation method has a simple structure and converges at an exponential rate \cite{Canuto:2012,Fornberg:1998,Trefethen:2000}.   The most well developed Gaussian quadrature methods are those that employ either Legendre-Gauss (LG) points \cite{Benson:2006,Rao:2010}, Legendre-Gauss-Radau (LGR) points \cite{Kameswaran:2008,Garg:2011a,Garg:2011b,Garg:2010}, or Legendre-Gauss-Lobatto (LGL) points \cite{Elnagar:1995}.  In addition, a convergence theory has recently been developed using Gaussian quadrature collocation.   Research on this theory had demonstrated that, under certain assumptions of the smoothness and coercivity, an $hp$ Gaussian quadrature method that employs either LG or LGR collocation points converges to a local minimizer of the optimal control problem \cite{Hager:2015a,Hager:2016a,Hager:2016b,Hager:2018,Hager:2019,DuChenHager:2019}.  In particular, it was shown in Refs.~\cite{Hager:2015a,Hager:2016a,Hager:2016b,Hager:2018,Hager:2019,DuChenHager:2019} that the convergence rate is exponentially fast as a function of the polynomial degree and is a polynomial function of the mesh interval width.  

While Gaussian quadrature orthogonal collocation methods are well suited to solving optimal control problems whose solutions are smooth, it is often the case that the solution of an optimal control problem has a nonsmooth optimal control \cite{br:aoc}. The difficulty in solving problems with nonsmooth control lies in determining when the nonsmoothness occurs. For example, dynamical systems where the control appears linearly or problems that have state inequality path constraints often have solutions where the control and state may be nonsmooth.  One approach to handling nonsmoothness is to employ a mesh refinement method where the optimal control problem is discretized using a sequence of meshes such that the last mesh satisfies a specified solution accuracy tolerance.  In the context of Gaussian quadrature collocation, $hp$ mesh refinement methods \cite{Kameswaran:2008,Garg:2011a,Garg:2011b,Garg:2010,Darby:2011a,Patterson:2015,Darby:2011,Liu:2015,Liu:2018} have been developed in order to improve accuracy in a wide variety of optimal control problems including those whose solutions are nonsmooth.  It is noted, however, that mesh refinement methods often place an unnecessarily large number of collocation points and mesh intervals near points of nonsmoothness in the solution.  Thus, it is beneficial to develop techniques that take advantage of the rapid convergence of a Gaussian quadrature collocation methods in segments where the solution is smooth and only increase the size of the mesh when necessary (thus, maintaining a smaller mesh than might be possible with a standard mesh refinement approach).  

For optimal control problems where the solution is nonsmooth the convergence theory developed in Refs.~\cite{Hager:2015a,Hager:2016a,Hager:2016b, Hager:2018,Hager:2019} is not applicable.  Consequently, when the solution of an optimal control problem is nonsmooth, an $hp$ method may not converge to a local minimizer of the optimal control problem.  A well studied class of problems where the smoothness and coercivity conditions found in Ref.~\cite{Hager:2016a} are not met are those where the control appears linearly in the problem formulation \cite{br:aoc,Guerra:2009,Li:2012,Martinon:2009}.  One approach for estimating the location of nonsmoothness is to introduce a variable called a break point \cite{Cuthrell:1987} that defines the location of nonsmoothness and to include this variable in the NLP.  The key problem that arises by introducing a break point is that the NLP has an extra degree of freedom.  As a result, the NLP may converge to a solution where this additional variable does not correspond to the location of the nonsmoothness.  This extra degree of freedom can be addressed by introducing additional constraints into the problem.  Reference \cite{Cuthrell:1989} introduced an inequality constraint at the break point location which bounded the value of each control at the break point.  However, in order to estimate the value of the control at the breakpoint, the control was parameterized as a polynomial using the information from the interior control points.  Research performed in Ref.~\cite{Chen:2014,Chen:2016,Chen:2019} introduced a bilevel or nested approach to solving nonsmooth optimal control problems. The bilevel approach formulates two separate NLP problems referred to as an inner problem and an outer problem.  The inner problem is the transcribed optimal control problem whereas the outer problem determines the properties of the mesh which is used by the inner problem and insures additional optimality conditions associated with nonsmooth optimal control problems are satisfied.  Similar to the method in Ref.~\cite{Cuthrell:1989}, the inner problem places explicit assumptions on the control functions. Reference~\cite{Ross:2004} also developed the concept of a knot using Legendre-Gauss-Lobatto collocation by introducing a variable that defines the switch time and collocating the dynamics at both the end of a mesh interval and the start of the subsequent mesh interval.  However, the LGL method used in in Ref.~\cite{Ross:2004} employs a square and singular differentiation matrix.  Therefore, unlike the approach of Ref.~\cite{Cuthrell:1987}, which used Legendre-Gauss collocation, the scheme used in Ref.~\cite{Ross:2004} is not a Gauss quadrature integrator.

The objective of this research is to develop a new method that employs Gaussian quadrature collocation and accurately approximates the solution of an optimal control problem whose solution is nonsmooth by letting the location of the nonsmoothness be a free variable in the problem.  In this paper, an approach is developed to improve upon the approach originally developed in Ref.~\cite{Cuthrell:1987} by gaining a better understanding why an incorrect location of the nonsmoothness in the optimal control is obtained when solving an optimal control problem using Legendre-Gauss-Radau collocation and introducing a constraint that will satisfy the equations of motion at the nonsmoothness but will not place explicit assumptions on the optimal control function.  Specifically, it is shown in this paper that the incorrect nonsmoothness location is obtained due to Lavrentiev phenomenon \cite{Lavrentiev:1927}.  Lavrentiev phenomenon occurs in a practical situation when it is desired to minimize a numerical approximation of a continuous (functional) optimization problem.  In particular, a continuous optimization problem may be subject to Lavrentiev phenomenon whenever a numerical approximation of a functional leads to an optimal objective value that is either strictly greater than or strictly less than the optimal value of the functional \cite{Ball:1987,Eide:2016,Eide:2018a,Eide:2018b}.  Simple examples of optimization problems that possess Lavrentiev phenomenon are given in Ref.~\cite{Mania:1934}, and the concept of Lavrentiev phenomenon has been extended to optimal control through the Lavrentiev gap \cite{Guerra:2009}.  The reason that the approximation of the continuous optimization problem has a higher or lower optimal objective arises from the possibility that the space over which the numerical optimization is performed may be different from the space over which the optimization needs to be performed in order to converge to the optimal solution.  Therefore, the existence and the behavior of Lavrentiev phenomenon depends upon the choice of the approximation method.  Moreover, any numerical scheme that gives rise to Lavrentiev phenomenon must somehow be augmented to compensate for any errors caused by the Lavrentiev phenomenon itself.  Initial explorations of Lavrentiev phenomenon using Gaussian quadrature collocation methods have been provided in Refs.~\cite{Eide:2016,Eide:2018a,Eide:2018b}.  In order to properly account for Lavrentiev phenomenon it is first necessary to understand the circumstances in which it occurs for any given numerical scheme.  

It is important to note that the approach developed in this paper is fundamentally different from the approaches developed in Refs.~\cite{Cuthrell:1987}, \cite{Ross:2004}, and \cite{Chen:2014,Chen:2016,Chen:2019}.  The key difference between the approach of this paper and that of Ref.~\cite{Cuthrell:1987} is that the search space is modified to include collocation constraints on the differential equations that are a function of control whereas the approach of Ref.~\cite{Cuthrell:1987} introduces no such additional collocation constraints.  Next, the key difference between the approach of this paper and the work of Ref.~\cite{Ross:2004} is that the work of Ref.~\cite{Ross:2004} collocates all of the differential equations at the end of a mesh interval where a solution may be nonsmooth whereas in this work collocation constraints are included at the end of a mesh interval on only those differential equations that are a function of control.  Second, the method of Ref.~\cite{Ross:2004} uses Legendre-Gauss-Lobatto which employs a square and singular differentiation matrix.  On the other hand, the approach developed in this paper employs Legendre-Gauss-Radau collocation where the differentiation matrix is rectangular.  Moreover, it has been shown previously that Legendre-Gauss-Radau is a Gaussian quadrature integration method \cite{Garg:2010}.  Finally, the key difference between the method of this paper and the methods of Refs.~\cite{Chen:2014,Chen:2016,Chen:2019} is that the methods of \cite{Chen:2014,Chen:2016,Chen:2019} parameterize the control as a function of time and this parameterization is used to approximate the control at the end of each mesh interval.  The method of this paper, however, introduces a variable that defines the control at the end of a mesh interval and adds collocation conditions at the end of the mesh interval using only those differential equations that are a function of the control.  

This paper presents a new method for Gaussian quadrature collocation.  In this new method, the standard LGR method is modified to include additional variables and additional constraints at the end of a mesh interval when compared with a previously developed standard Legendre-Gauss-Radau collocation method.  The additional variables are the time associated with mesh interval boundaries and the corresponding value of the control at the end of the mesh interval.  The additional constraints are collocation conditions on those differential equations that are a function of the control and inequality constraints on the control at the endpoint of each mesh interval.  It is important to note that the additional constraints are added to only those collocation constraints associated with the differential equations that are functions of the control and are {\em not} added to all differential equations.  The modified method results in a different control variable at the end of each mesh interval from the control variable at the start of the next mesh interval.  A costate estimation method is then developed that transforms the Lagrange multipliers of the NLP to the costate of the optimal control problem \cite{Garg:2011a,Garg:2011b,Garg:2010}.  Using this costate estimation method, the transformed adjoint system \cite{Garg:2011a,Garg:2011b,Garg:2010,Hager:2000} of the modified LGR collocation method is developed.  It is also shown that the state and control obtained from the modified LGR method along with the new costate estimation scheme satisfies one of the necessary Weierstrass-Erdmann conditions when the solution of the optimal control problem is nonsmooth and therefore does not require additional constraints to enforce the Weierstrass-Erdmann conditions.

The remainder of this paper is organized as follows.  Section \ref{sect:notation} provides the notations and conventions used in this paper.  Section \ref{sect:bolza} present the Bolza optimal control problem.  Section \ref{sect:LGR} presents the standard Legendre-Gauss-Radau (LGR) collocation method for discretizing optimal control problems.  Section \ref{sect:Lavrentiev-LGR} provides a description of Lavrentiev phenomenon, a discussion of the Lavrentiev gap that arises when using LGR collocation to solve an optimal control problem whose solution is nonsmooth, and an analysis of the search space using LGR collocation.  Section \ref{sect:LGRmod} presents the modified LGR collocation method for solving optimal control problems with nonsmooth solutions.  Section \ref{sect:modadjoint} derives the transformed adjoint system along with the Weierstrass-Erdmann conditions that arise from modified LGR collocation and demonstrates the accuracy of the modified LGR collocation method costate estimate.  Finally, Section \ref{sect:conclusions} provides conclusions on this research.  

\section{Notation and Conventions\label{sect:notation}}

In this paper, the following notation and conventions will be used.  First, the independent variable is denoted $\tau$.  Therefore, the notation $x(\tau)$ denotes a dependence of the quantity $x$ on $\tau$.  Next, all vectors will be denoted as {\em row} vectors.  Therefore, if $\m{x}(\tau)\in\bb{R}^n$ is a vector function of $\tau$, then $\m{x}(\tau)$ is given as
\begin{equation}\label{convention-row-vector-function-of-tau}
  \m{x}(\tau) = \left[x_1(\tau),x_2(\tau),\ldots,x_n(\tau)\right].
\end{equation}
Suppose now that $\m{x}(\tau)$ is approximated using a basis of Lagrange polynomials $\ell_j(\tau),\; (j=1,\ldots,N+1)$ as
\begin{equation}\label{convention-lagrange-approximation-x}
  \m{x}(\tau) \approx \hat{\m{x}}(\tau) = \sum_{j=1}^{N+1} \m{X}_j\ell_j(\tau), \quad   \ell_j(\tau) = \prod_{\substack{l=1\\l\neq j}}^{N+1} \frac{\tau-\tau_l}{\tau_j-\tau_l},
\end{equation}
where $(\tau_1,\ldots,\tau_{N+1})$ are the support points of $\ell_j(\tau),\;(j=1,\ldots,N+1)$.  It is known that the Lagrange polynomials $\ell_j(\tau),\;(j=1,\ldots,N+1)$ satisfy the property
\begin{equation}\label{convention-lagrange-basis-isolation-property}
  \ell_j(\tau_i) = \delta_{ij} = \left\{\begin{array}{lcl} 1 & , & i=j, \\ 0 & , & i\neq j. \end{array} \right. 
\end{equation}
which implies that
\begin{equation}\label{convention-X(tauk)=Xk}
  \m{X}(\tau_i) = \m{X}_i, \quad (i=1,\ldots,N+1).  
\end{equation}
Using aforementioned row vector conventions and function approximations, in this paper the notation $\m{X}_{i:j}$ is a matrix whose rows are the values $(\m{X}_i,\ldots,\m{X}_j)$, that is,
\begin{equation}\label{convention-Xij}
  \m{X}_{i:j} = \left[\begin{array}{c} \m{X}_i \\ \m{X}_{i+1} \\ \vdots \\ \m{X}_j \end{array} \right]
\end{equation}
Furthermore, the notation $\m{A}\tr$ denotes the transpose of a matrix $\m{A}$.  The inner product between two matrices $\m{A}$ and $\m{B}$ of the same size is then denoted $\langle \m{A}, \m{B} \rangle$ and is defined as
\begin{equation}\label{convention-inner-product}
  \langle \m{A}, \m{B} \rangle = \mbox{trace } \m{A}\tr \m{B} .
\end{equation}
Note that, when $\m{A}$ and $\m{B}$ are row vectors, $\langle \m{A}, \m{B} \rangle$ is the standard inner product.

Next, differentiation matrices are used throughout this paper.  % In order to properly distinguish referring to particular elements, including rows or columns of a differentiation matrix
The following conventions will be adopted for the elements of a differentiation matrix $\m{D}$:
\begin{displaymath}
  \begin{array}{lcl}
    \m{D}_{(i,j)} & = & \textrm{element in row }i\textrm{ and column }j, \\
    \m{D}_{(:,i)} & = & \textrm{elements in all rows and column }i, \\ 
    \m{D}_{(i,:)} & = & \textrm{elements all columns and row }i, \\
    \m{D}_{(i:j,k:l)} & = & \textrm{elements in rows }i\textrm{ through }j\textrm{ and columns }k\textrm{ through }l.
  \end{array}
\end{displaymath}
Finally, the following conventions are adopted for functions and their first derivatives (gradients or Jacobians).  First, if $\m{f} : \mathbb{R}^n \rightarrow \mathbb{R}^m$ is a function of the vector $\m{x}\in\mathbb{R}^n$, then $\m{f}(\m{x})$ is given as
\begin{equation}\label{convention-vector-f}
  \m{f}(\m{x}) = \left[f_1(\m{x}),f_2(\m{x}),\ldots,f_m(\m{x})\right]
\end{equation}
Furthermore, the notation $\nabla_{\m{x}} \m{f}(\m{x})$ is defined as
\begin{equation}\label{convention-nabla-f}
  \nabla\m{f}(\m{x}) = \frac{\partial\m{f}}{\partial\m{x}} = \left[\begin{array}{c} \frac{\partial f_1}{\partial \m{x}} \\ \frac{\partial f_2}{\partial\m{x}} \\ \vdots \\\frac{\partial f_m}{\partial\m{x}} \end{array} \right]
  =
  \left[
    \begin{array}{ccc}
      \frac{\partial f_1}{\partial x_1} &   \ldots & \frac{\partial f_1}{\partial x_n} \\
      \frac{\partial f_2}{\partial x_1} &   \ldots & \frac{\partial f_2}{\partial x_n} \\
      \vdots & \ddots & \vdots \\
      \frac{\partial f_m}{\partial x_1} &   \ldots & \frac{\partial f_m}{\partial x_n}
    \end{array}
  \right]
\end{equation}
Then, using the definitions provided in Eqs.~\eqref{convention-vector-f} and \eqref{convention-nabla-f}, if $g:\mathbb{R}^{n\times m}\rightarrow\mathbb{R}$ is a scalar function of the $m\times n$ matrix $\m{X}\in\mathbb{R}^{m\times n}$, then the gradient of $g(\m{X})$ with respect to $\m{X}$, denoted $\nabla_{\m{X}}~g(\m{X})$, is defined as
\begin{equation}
  \nabla_{\m{X}}~g(\m{X}) =
  \left[
    \begin{array}{ccc}
      \frac{\partial g}{\partial  X_{11}} & \ldots & \frac{\partial g}{\partial  X_{1n}} \\ 
      \frac{\partial g}{\partial  X_{21}} & \ldots & \frac{\partial g}{\partial  X_{2n}} \\
      \vdots & \ddots & \vdots \\
      \frac{\partial g}{\partial  X_{m1}} & \ldots & \frac{\partial g}{\partial  X_{mn}}
    \end{array}
  \right].
\end{equation}

\section{Bolza Optimal Control Problem\label{sect:bolza}}

Without loss of generality, consider the following optimal control problem in Bolza form.  Minimize the objective functional 
\begin{equation} \label{eqn:contcost}
  \mathcal{J}  = \mathcal{M}(\m{x}(-1),\m{v}(-1),\m{x}(+1),\m{v}(+1),t_0,t_f) + \frac{t_f - t_0}{2}\int_{-1}^{+1} \mathcal{L}(\m{x}(t),\m{v}(t),\m{u}(t)) dt,
\end{equation}
subject to the dynamic constraints
\begin{equation} \label{eqn:contdyn}
  \begin{array}{lcl}
    \displaystyle\frac{d\m{x}(t)}{dt} & = & \displaystyle \frac{t_f-t_0}{2}\m{f}_x(\m{x}(t),\m{v}(t)), \\
    \displaystyle\frac{d\m{v}(t)}{dt} & = & \displaystyle \frac{t_f - t_0}{2}\m{f}_v(\m{x}(t),\m{v}(t),\m{u}(t)),
  \end{array}
\end{equation}
inequality path constraints
\begin{equation} \label{eqn:contpath}
  \m{c}(\m{x}(t),\m{v}(t),\m{u}(t)) \leq \m{0},
\end{equation}
and the boundary conditions
\begin{equation} \label{eqn:contbound}
  \m{b}(\m{x}(-1),\m{v}(-1),\m{x}(+1),\m{v}(+1),t_0,t_f) = \m{0}.
\end{equation}
It is noted in Eqs.~\eqref{eqn:contcost}--\eqref{eqn:contbound} that $\m{x}(t)\in\R^{n_x}$, $\m{v}(t)\in\R^{n_v}$, and, together, $(\m{x}(t),\m{v}(t))\in\R^n$ is the state (where $n=n_x+n_v$), $\m{u}(t) \in \R^{n_u}$ is the control, $\m{f}_x: \R^{n_x}\times\R^{n_v}\rightarrow\R^{n_x}$, $\m{f}_v: \R^{n_x}\times\R^{n_v}\times\R^{n_u}\rightarrow\R^{n_v}$, $\m{c}: \R^{n_x}\times\R^{n_v}\times \R^{n_u} \rightarrow \R^{n_c}$, $\m{b}: \R^{n_x}\times\R^{n_v}\times \R^{n_x}\times\R^{n_v} \rightarrow \R^{n_b}\times \R \times \R$, $\mathcal{M}: \R^{n_x}\times\R^{n_v}\times \R^{n_x}\times\R^{n_v} \rightarrow \R$, and $\mathcal{L}: \R^{n_x}\times\R^{n_v}\times \R^{n_u}\rightarrow \R \times \R \times \R$.  It is seen from from the optimal control problem in Eq.~\eqref{eqn:contcost}--\eqref{eqn:contbound} that the dynamics are decomposed into those differential equations that depend upon the control and those differential equations that do not depend upon the control.  This decomposition is done deliberately because the modified Legendre-Gauss-Radau collocation method developed in this paper exploits this separation.  It is noted that no generality is lost with such a decomposition because $(n_x,n_v)=(0,n)$ is a special case of the dynamics given in Eq.~\eqref{eqn:contdyn}.

Consider now the following partitioning of the independent variable $t\in[-1,+1]$ into a mesh consisting of $K+1$ mesh points $-1=T_0<T_1<T_2<\ldots<T_K=+1$, where $\mathcal{I}_k=[T_{k-1},T_k]$ corresponds to mesh interval $k\in[1,\ldots,K]$.  Then the Bolza optimal control problem of Section \ref{sect:bolza} can be expressed in multiple-interval form as follows.  Minimize the objective functional
\begin{equation} \label{eqn:contcost-mi}
  \mathcal{J}  = \mathcal{M}(\m{x}^{(1)}(T_0),\m{v}^{(1)}(T_0),\m{x}^{(K)}(T_K),\m{v}^{(K)}(T_K),t_0,t_f) + \frac{t_f - t_0}{2}\sum_{k=1}^{K} \int_{T_{k-1}}^{T_k} \mathcal{L}(\m{x}\rk(t),\m{v}\rk(t),\m{u}\rk(t)) dt,  
\end{equation} 
subject to the dynamic constraints  
\begin{equation} \label{eqn:contdyn-mi}
  \begin{array}{lcl}
    \displaystyle\frac{d\m{x}^{(k)}(t)}{dt} & = & \displaystyle \frac{t_f-t_0}{2}\m{f}_x(\m{x}^{(k)}(t),\m{v}^{(k)}(t)), \\
    \displaystyle\frac{d\m{v}^{(k)}(t)}{dt} & = & \displaystyle \frac{t_f- t_0}{2}\m{f}_v(\m{x}^{(k)}(t),\m{v}^{(k)}(t),\m{u}^{(k)}(t)),  
  \end{array}
\end{equation} 
inequality path constraints  
\begin{equation} \label{eqn:contpath-mi}
  \m{c}(\m{x}^{(k)}(t),\m{v}^{(k)}(t),\m{u}^{(k)}(t)) \leq \m{0},  \quad (k=1,\ldots,K),
\end{equation} 
the boundary conditions  
\begin{equation} \label{eqn:contbound-mi}
  \m{b}(\m{x}^{(1)}(T_0),\m{v}^{(1)}(T_0),\m{x}^{(K)}(T_K),\m{v}^{(K)}(T_K),t_0,t_f) = \m{0},
\end{equation} 
and the state continuity constraint
\begin{equation} \label{eqn:state-continuity}
  \left(\m{x}^{(k)}\left(T_{k}\right),\m{v}^{(k)}\left(T_{k}\right)\right) = \left(\m{x}^{(k+1)}\left(T_{k}\right), \m{v}^{(k+1)}\left(T_{k}\right)\right), \quad (k=1,\ldots,K-1)
\end{equation} 
at the boundaries of the interior mesh intervals.  It is noted that Eq.~\eqref{eqn:state-continuity} ensures continuity in the constraint across the domain $t\in[-1,+1]$ as is assumed in the original formulation of the Bolza optimal control problem stated in Section \ref{sect:bolza}.

The multiple-interval form of the Bolza optimal control problem given in Eqs.~\eqref{eqn:contcost-mi}--\eqref{eqn:state-continuity} is now transformed to the independent variable $\tau\in[-1,+1]$ on each mesh interval $\mathcal{I}_k,\; (k=1,\ldots,K)$.  First, it is seen that $t\in[T_{k-1},T_k]$ can be related to $\tau\in[-1,+1]$ as
\begin{equation}\label{eqn:s-in-terms-of-tau}
  t = \frac{T_k-T_{k-1}}{2}\tau + \frac{T_k+T_{k-1}}{2}
\end{equation}
which further implies that
\begin{equation}\label{eqn:dsdtau}
  \frac{dt}{d\tau} = \frac{T_k-T_{k-1}}{2} \equiv \alpha_k,\quad (k=1,\ldots,K). 
\end{equation}
Consequently, the multiple-interval Bolza optimal control problem given in Eqs.~\eqref{eqn:contcost-mi}--\eqref{eqn:state-continuity} can be written in terms of the variable $\tau$ as follows.  Minimize the objective functional
\begin{equation} \label{eqn:contcost-mi-tau}
  \mathcal{J}  = \mathcal{M}(\m{x}^{(1)}(-1),\m{v}^{(1)}(-1),\m{x}^{(K)}(+1),\m{v}^{(K)}(+1),t_0,t_f) + \frac{t_f - t_0}{2}\sum_{k=1}^{K}\int_{-1}^{+1} \alpha_k \mathcal{L}(\m{x}\rk(\tau),\m{v}\rk(\tau),\m{u}\rk(\tau)) d\tau, 
\end{equation}
subject to the dynamic constraints 
\begin{equation} \label{eqn:contdyn-mi-tau}
  \begin{array}{lclcl}
    \displaystyle\frac{d\m{x}^{(k)}(\tau)}{d\tau} & \equiv & \dot{\m{x}}^{(k)}(\tau) & = & \displaystyle \frac{t_f-t_0}{2}\alpha_k\m{f}_x(\m{x}^{(k)}(\tau),\m{v}^{(k)}(\tau)), \\
    \displaystyle\frac{d\m{v}^{(k)}(\tau)}{d\tau} & \equiv & \dot{\m{x}}^{(k)}(\tau) & = & \displaystyle \frac{t_f - t_0}{2}\alpha_k\m{f}_v(\m{x}^{(k)}(\tau),\m{v}^{(k)}(\tau),\m{u}^{(k)}(\tau)), 
  \end{array}, \quad (k=1,\ldots,K),
\end{equation}
inequality path constraints 
\begin{equation} \label{eqn:contpath-mi-tau}
  \m{c}(\m{x}^{(k)}(\tau),\m{v}^{(k)}(\tau),\m{u}^{(k)}(\tau)) \leq \m{0}, \quad (k=1,\ldots,K),
\end{equation}
the boundary conditions 
\begin{equation} \label{eqn:contbound-mi-tau}
  \m{b}(\m{x}^{(1)}(-1),\m{v}^{(1)}(-1),\m{x}^{(K)}(+1),\m{v}^{(K)}(+1),t_0,t_f) = \m{0}, 
\end{equation}
the state continuity constraint 
\begin{equation} \label{eqn:state-continuity-tau}
  \left(\m{x}^{(k)}(+1), \m{x}^{(k)}(+1)\right) = \left(\m{x}^{(k+1)}(-1), \m{x}^{(k+1)}(-1)\right), \quad (k=1,\ldots,K-1),
\end{equation}

\section{Legendre-Gauss-Radau Collocation\label{sect:LGR}}

The Legendre-Gauss-Radau (LGR) collocation method approximates the multiple-interval form of the Bolza optimal control problem defined in Section \ref{sect:bolza}.  First, it is assumed that the number of collocation points is the same in each mesh interval and is denoted $N$.  Next, let $\tau_i,\;(i=1,\ldots,N)$ be the $N$ Legendre-Gauss-Radau collocation points \cite{Abramowitz1} on the interval $[-1,+1)$ and that $\tau_{N+1}=+1$ is a noncollocated point.  Then, in every mesh interval $k\in[1,\ldots,K]$, the state $(\m{x}^{(k)}(\tau),\m{v}^{(k)}(\tau))$ is approximated as
\begin{equation}\label{state-approximation}
  \begin{array}{lcl}
    \m{x}^{(k)}(\tau) & \approx & \hat{\m{x}}^{(k)}(\tau) = \sum_{j=1}^{N+1} \m{X}_{j}^{(k)} \ell_j (\tau), \\
    \m{v}^{(k)}(\tau) & \approx & \hat{\m{v}}^{(k)}(\tau) = \sum_{j=1}^{N+1} \m{V}_{j}^{(k)} \ell_j (\tau), \\
  \end{array}
\end{equation}
where $\ell_j(\tau)$ are the Lagrange polynomials
\begin{equation}\label{eqn:lagrange-basis-lgr}
  \ell_j(\tau)  = \displaystyle \prod_{\substack{l=1\\j\neq l}}^{N+1} \frac{\tau-\tau_l}{\tau_j-\tau_l}, \qquad \left(j=1,\ldots,N+1\right)
\end{equation}
whose support points are $(\tau_1,\ldots,\tau_{N+1})$.  Differentiating $\m{x}^{(k)}(\tau)$ and $\m{v}^{(k)}(\tau)$ in Eq.~\eqref{state-approximation} gives
\begin{equation}\label{state-derivative-approximation}
  \begin{array}{lcl}
    \dot{\m{x}}^{(k)}(\tau) & \approx & \dot{\hat{\m{x}}}^{(k)}(\tau) = \sum_{j=1}^{N+1} \m{X}_{j}^{(k)} \dot{\ell}_j(\tau), \\
    \dot{\m{v}}^{(k)}(\tau) & \approx & \dot{\hat{\m{v}}}^{(k)}(\tau) = \sum_{j=1}^{N+1} \m{V}_{j}^{(k)} \dot{\ell}_j(\tau). \\
  \end{array}
\end{equation}
Evaluating the functions $\dot{\m{x}}^{(k)}(\tau)$ and $\dot{\m{v}}^{(k)}(\tau)$ at $\tau=\tau_i$ gives
\begin{equation}\label{state-derivative-approximation-at-tauk}
  \begin{array}{lcl}
    \dot{\m{x}}^{(k)}(\tau_i) & \approx & \dot{\hat{\m{x}}}^{(k)}(\tau_i) = \sum_{j=1}^{N+1} \m{X}_j ^{(k)} \dot{\ell}_j(\tau_i) = \sum_{j=1}^{N+1} \m{D}_{(i,j)} \m{X}_{j}^{(k)} , \\
    \dot{\m{v}}^{(k)}(\tau_i) & \approx & \dot{\hat{\m{v}}}^{(k)}(\tau_i) = \sum_{j=1}^{N+1} \m{V}_j ^{(k)} \dot{\ell}_j(\tau_i) = \sum_{j=1}^{N+1} \m{D}_{(i,j)} \m{V}_{j}^{(k)},
  \end{array}
\end{equation}
where the coefficients $\m{D}_{(i,j)},\;(i=1,\ldots,N; \; j=1,\ldots,N+1)$ form the $N\times (N+1)$ {\em LGR differentiation matrix} $\m{D}$.  Next, the matrices $\m{X}^{(k)}\in\R^{(N+1)\times n_x}$ and $\m{V}^{(k)}\in\R^{(N+1)\times n_v}$ correspond row-wise to the state approximations at $(\tau_1,\ldots,\tau_{N+1})$, while the matrix $\m{U}^{(k)}\in\R^{N\times n_u}$ corresponds row-wise to the approximations of the control at $(\tau_1,\ldots,\tau_N)$.  The LGR approximation of the state leads to the following nonlinear programming problem (NLP) that approximates the optimal control problem given in Eqs.~\eqref{eqn:contcost}--\eqref{eqn:contbound}.  Minimize the objective function
\begin{equation}\label{eqn:LGRcost}
  J = \mathcal{M}(\m{X}_1^{(1)},\m{V}_1^{(1)},\m{X}_{N+1}^{(K)},\m{V}_{N+1}^{(K)},t_0,t_f) +\frac{t_f-t_0}{2}\sum_{k=1}^{K}\sum_{i=1}^{N} \alpha_k w_i \mathcal{L}(\m{X}_i^{(k)},\m{V}_i^{(k)},\bmU_i^{(k)}),
\end{equation}
subject to 
\begin{equation} \label{eqn:LGRcon}
  \begin{array}{lcl}
    \m{D}_{(i,:)}\m{X}^{(k)} - \frac{t_f-t_0}{2}\alpha_k\m{f}_x\left(\m{X}_i^{(k)},\m{V}_i^{(k)}\right) &= &  \m{0}, \\
    \m{D}_{(i,:)}\m{V}^{(k)} - \frac{t_f-t_0}{2}\alpha_k\m{f}_v(\m{X}_i^{(k)},\m{V}_i^{(k)},\m{U}_i^{(k)}) &=&  \m{0}, \quad (k=1,\ldots,K), \\
    \m{c}({\m{X}_i^{(k)},\m{V}_i^{(k)},\m{U}_i^{(k)}}) &\leq &  \m{0}, \\
    \m{b}(\m{X}_1^{(1)},\m{V}_1^{(1)},\m{X}_{N+1}^{(K)},\m{V}_{N+1}^{(K)},t_0, t_f) & \leq & \m{0}, \\
    \left(\m{X}_{N+1}^{(k)}, \m{V}_{N+1}^{(k)}\right)  & = & \left(\m{X}_{1}^{(k+1)}, \m{V}_{1}^{(k+1)}\right), \quad (k=1,\ldots,K-1),
  \end{array}
\end{equation}
where $i\in(1,\ldots,N)$.  It is noted in Eq.~\eqref{eqn:LGRcost} that $w_i,\;\; (i = 1,\ldots,N)$ are the LGR quadrature weights.  Equations~\eqref{eqn:LGRcost} and \eqref{eqn:LGRcon} will be referred to as the {\em Legendre-Gauss-Radau collocation method}. 

\section{LGR Collocation and Lavrentiev Phenomenon\label{sect:Lavrentiev-LGR}}

This section provides an overview of Lavrentiev phenomenon.  First, Section \ref{sect:Lavrentiev} provides a discussion of the concept of Lavrentiev phenomenon and how Lavrentiev phenomenon manifests itself when using LGR collocation.  Next, Section \ref{sect:LavrentievGap} provides a discussion of the Lavrentiev gap that arises when solving an optimal control problem whose solution is nonsmooth using LGR collocation.  Finally, Section \ref{sect:SearchSpace-LGR} provide an analysis of the search space using LGR collocation on an example optimal control problem whose solution is nonsmooth.  

\subsection{Lavrentiev Phenomenon\label{sect:Lavrentiev}}

The LGR collocation method described in Section \ref{sect:LGR} is a finite element method that approximates the optimal control problem described in Section \ref{sect:bolza} with a finite-dimensional nonlinear programming problem (NLP).  Ref.~\cite{Ball:1987} described cases where the optimal solution of the finite-dimensional approximation produces a pseudo-minimizer that differs from the true optimal solution.  The behavior described in Ref.~\cite{Ball:1987} is called {\em Lavrentiev phenomenon} \cite{Lavrentiev:1927,Ball:1987} and is important to understand and address when solving an optimal control problem using LGR collocation.  To gain an understanding of Lavrentiev phenomenon, consider the classical least action calculus of variations problem of the form
\begin{equation}\label{eqn:cov-problem}
  \textrm{min } J(x) = \int_{a}^{b} L(x(t),\dot{x}(t),t)dt,\textrm{ subject to }(x(a),x(b))=(x_0,x_f), 
\end{equation} 
Suppose now that $\mathcal{A}(a,b)$ and $\mathcal{W}(a,b)$ are, respectively, the space of absolutely continuous functions and Lipschitz continuous functions on the interval $t\in[a,b]$.  Furthermore, consider particular instances where the minimizer $x^*(t)$ lies in $\mathcal{A}(a,b)$ \cite{Ball:1987}.  For such cases, the minimizer $x^*(t)$ of $J(x)$ has an unbounded derivative at certain points \cite{Ball:1987}, and these singularities may prevent the minimizer from satisfying the classical first-order Euler-Lagrange necessary optimality conditions
\begin{equation}\label{eqn:Euler-Lagrange}
  \frac{\partial L}{\partial x} - \frac{d}{dt}\frac{\partial L}{\partial\dot{x}} = 0,  
\end{equation}
where, in general, the weak form of the Euler-Lagrange equations that is usually satisfied.  

Now, in general it is not possible to solve a calculus of variations problem analytically.  Consequently, the integral in Eq.~\eqref{eqn:cov-problem} must be approximated numerically via quadrature using a finite-element method, and this quadrature approximation leads to a finite-dimensional nonlinear programming problem (NLP) that must be solved using nonlinear optimization solvers \cite{Gill:2005,Biegler:2008}.  As it turns out, when the minimizer $x^*(t)$ lies in $\mathcal{A}(a,b)$ collocation methods typically fail in computing both the correct minimizer $x^*(t)$ and the correct minimizing value of the integral $I(x)$.  To illustrate the failure of the finite element method, consider the following problem \cite{Ball:1987} of minimizing over $\mathcal{A}(0,1)$ the integral
% The failure of a finite element method consider the problem of minimizing over $\mathcal{A}(0,1)$ the integral 
\begin{equation}\label{eqn:LavExample}
  J(x) = \int_{0}^1(x^3(t)-t)^2\dot{x}^6(t)dt\quad,\quad (x(0),x(+1)) = (0,+1), 
\end{equation}
where $x^*(t)=t^{1/3}\in\mathcal{A}(0,1)$ is the unique minimizer and $I(x^*)=0$.  This last fact, namely that $x^*(t)$ lies in $\mathcal{A}(0,1)$, can be connected to the following result from Mani\'{a} \cite{Mania:1934}:
\begin{equation}\label{eqn:LavPhenomenon}
  \inf_{x \in\mathcal{W}(0,1)} J(x) > \inf_{x\in\mathcal{A}(0,1)} J(x) = 0.
\end{equation}
The property defined by Eq.~\eqref{eqn:LavPhenomenon} is called {\em Lavrentiev phenomenon} and shows that the minimizer over $\mathcal{A}(0,1)$ differs from the minimizer over $\mathcal{W}(0,1)$.  Thus, if the minimizer is absolutely continuous while the optimization search is performed over the space of Lipschitz continuous functions, the result of the optimization will be a pseudo-minimizer $\bar{x}\neq x^*$ \cite{Ball:1987}.

The preceding discussion leads into the fact that Lavrentiev phenomenon can create misleading results when employing numerical optimization with a finite-element method.  To see the effect that Lavrentiev phenomenon can have when using a finite element method, consider the following equivalent formulation of the problem in Eq.~\eqref{eqn:LavExample} as the Lagrange optimal control problem
\begin{equation}\label{eqn:LavExampleOC}
  \textrm{min }J(x,u) = \int_{0}^{+1}(x^3(t)-t)^2 u^6(t)dt\quad \textrm{subject to }\left\{
    \begin{array}{lcl}
      \dot{x}(t) & = & u(t), \\
      (x(0),x(+1)) & = & (0,1).
    \end{array}
  \right.
\end{equation}
Figure~\ref{fig:LavExampleOCSolution} shows the exact solution $x^*(t)$ alongside the solution obtained using the multiple-interval LGR collocation method \cite{Kameswaran:2008,Garg:2011a,Garg:2011b,Garg:2010,Darby:2011a,Darby:2011b,Patterson:2015,Liu:2015,Liu:2018} described in Section \ref{sect:LGR} using $N=4$ LGR collocation points in each mesh interval and the NLP solver IPOPT \cite{Biegler:2008}.  Similar to the result obtained in Ref.~\cite{Ball:1987} using midpoint rule integration, Fig.~\ref{fig:LavExampleOCSolution} shows that the LGR approximation does not match the optimal solution $x^*(t)$.  In fact, consistent with the discussion in Ref.~\cite{Ball:1987}, the LGR approximation converges to pseudo-minimizer that differs from $x^*(t)$.  Over the years, the concept of Lavrentiev phenomenon has been expanded beyond those that involve the space of absolutely continuous and Lipschitz continuous functions \cite{Ferriero:2004}.  For instance, Guerra \cite{Guerra:2009} examined the space of a singular arc optimal control problem against the space of the optimal control problem created when the singular problem is regularized.

\begin{figure}[ht]
  \centering
  \includegraphics[scale=0.45]{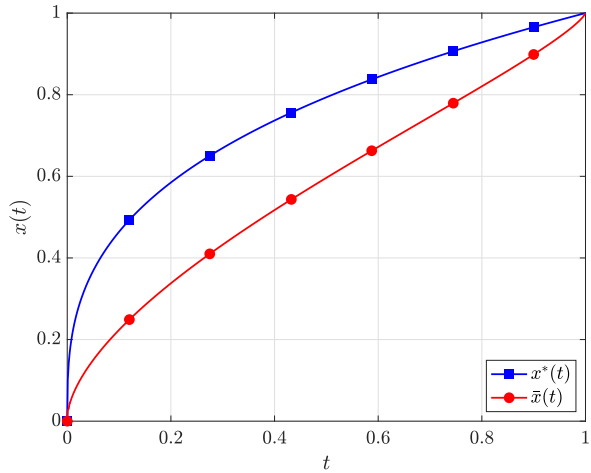}
  \caption{Minimizer $x^*(t)=t^{1/3}$ of Eq.~\eqref{eqn:LavExampleOC} alongside $hp$ LGR $(N=4)$ collocation pseudo-minimizer $\bar{x}(t)$.\label{fig:LavExampleOCSolution}}
\end{figure}

The modification of the LGR collocation method described in Section \ref{sect:LGR} developed in this paper is motivated by the preceding discussion of Lavrentiev phenomenon.  In particular, it was shown in this section that it is essential to perform the optimization over the appropriate search space.  The method developed in this paper focuses on identifying the correct search space when solving an optimal control problem whose solution is nonsmooth.  The remainder of this paper focuses on the aforementioned modification of the LGR collocation method \cite{Kameswaran:2008,Garg:2011a,Garg:2011b,Garg:2010,Darby:2011a,Darby:2011b,Patterson:2015,Liu:2015,Liu:2018} described in Section \ref{sect:LGR}.  

\subsection{Lavrentiev Gap\label{sect:LavrentievGap}}

The main point of Section \ref{sect:Lavrentiev} is that the search space of the optimization problem affects the solution obtained.  If an incorrect search space is used, an incorrect solution may be obtained and the corresponding objective may be either smaller or larger than the optimal objective.  The difference between the incorrect and correct search spaces is called the {\em Lavrentiev gap}.  If the incorrect search space is strictly larger than the correct search space, then the Lavrentiev gap is said to be negative.  On the other hand, if the incorrect search space is strictly smaller than the correct search space, then the Lavrentiev gap is said to be positive.  

This research focuses on the solution of optimal control problems with nonsmooth solutions.  For such problems, it is desired to improve the accuracy in the numerical solution by adjusting the mesh points to coincide with the locations of nonsmoothness in the solution.  In particular, if the mesh points are positioned at the exact locations of nonsmoothness and the solution is smooth on the interior of each mesh interval, then the numerical approximation of the optimal control problem would be smooth.

Typically, the locations of nonsmoothness in the solution of an optimal control problem are not known a priori.  One strategy for computing a numerical approximation of the solution to the optimal control problem is introduce variables in the optimization that correspond to the locations of nonsmoothness and then adjust the values of these variables to improve the accuracy of the approximation.  As will be shown in Section \ref{sect:SearchSpace-LGR}, the idea of adding variables that correspond to the locations of nonsmoothness may lead to an incorrect result because the search space may be larger than the correct search space.  In such a case, the Lavrentiev gap is negative and convergence to an objective value smaller than the optimal value occurs.  To close the gap, additional constraints are required to reduce the size of the search space.    In this paper it is shown that adding a new collocation constraint closes the Lavrentiev gap.  

\subsection{Analysis of Search Space Using LGR Collocation\label{sect:SearchSpace-LGR}}

To show the occurrence of Lavrentiev phenomenon \cite{Ball:1987,Mania:1934} as described in Section \ref{sect:Lavrentiev} and the Lavrentiev gap as described in Section \ref{sect:LavrentievGap}, in this section the search space associated with the LGR collocation method developed in Section \ref{sect:LGR} is analyzed using an example whose solution contains a bang-bang optimal control.  The results obtained studying this bang-bang optimal control problem then sets the stage for the modified LGR collocation method developed in Section \ref{sect:LGRmod}.

\subsection{Motivating Example\label{sect:example}}

Consider the following optimal control problem:
\begin{equation}\label{eqn:example}
  \textrm{min } t_f \textrm{ subject to }
  \left\{
    \begin{array}{lcr}
      \left(\dot{x}(t),\dot{v}(t)\right) & = & \displaystyle\frac{t_f}{2}\left(v(t),u(t)\right), \\
      u(t) & \in & (u_{\min},u_{\max}) = (-1,+1), \\
      (x(-1),x(+1),v(-1),v(+1)) & = & (x_0,v_0,x_f,v_f) = (10,0,0,0).
    \end{array}
  \right.
\end{equation}
The optimal solution to the optimal control problem given in Eq.~\eqref{eqn:example} is
\begin{equation}\label{eqn:disolpos}
  (x^*(t),v^*(t),u^*(t)) =
  \left\{
    \begin{array}{lcl}
      \displaystyle\left(-x_0\frac{(t+1)^2}{2}+x_0,-\sqrt{x_0}(t+1),-1\right)  & , & t\leq t_s, \vspace{3pt} \\
      \displaystyle\left(+x_0\frac{(t-1)^2}{2},+\sqrt{x_0}(t-1),+1\right)  & , & t> t_s, \vspace{3pt} \\
    \end{array}
  \right.
\end{equation}
where $t_s^*= 0$ and $t_f^* = 2\sqrt{x_0}\approx 6.32456$.  It is seen that the $x^*(t)$ is piecewise quadratic while $u^*(t)$ is bang-bang with a single switch at $t = t_s^*=0$.  Given that the state is piecewise quadratic, it should be possible to obtain the exact solution to this example using two intervals.

\subsection{Two-Interval Reformulation of Example:  Lavrentiev Gap\label{sect:example-two-interval}}

Consider now the following two-interval reformulation of the example given in Section \ref{sect:example}:
\begin{equation}\label{eqn:exampleMI}
  \textrm{min } t_f \textrm{ subject to }
  \left\{
    \begin{array}{lcr}
      \left(\dot{x}^{(k)}(\tau),\dot{v}^{(k)}(\tau)\right) & = & \displaystyle\frac{t_f}{2}\alpha_k\left(v ^{(k)}(\tau),u ^{(k)}(\tau)\right), \\
      u^{(k)}(\tau) & \in & (u_{\min},u_{\max}) = (-1,+1), \\
      (x^{(1)}(-1),x^{(2)}(+1),v^{(1)}(-1),v^{(2)}(+1)) & = & (10,0,0,0).
    \end{array}
  \right. \quad (k=1,2)
\end{equation}
where $\alpha_k=(T_k-T_{k-1})/2,\; (k=1,2)$ as given in Eq.~\eqref{eqn:dsdtau} and $T_1$ is a variable in the problem formulation of Eq.~\eqref{eqn:exampleMI} and represents the time at the boundary between the two mesh intervals $\mathcal{I}_1$ and $\mathcal{I}_2$.  Suppose now that the LGR collocation method is used to approximate the two-interval optimal control problem of Eq.~\eqref{eqn:exampleMI}.  Because the optimal trajectory is piecewise quadratic and the LGR quadrature is exact for polynomials of degree at most $2N-2$, it should be possible to obtain the exact solution using $N=2$ collocation points in each subinterval with $T_1$ included as an optimization variable.

Now define the {\em approximate control} as
\begin{equation}\label{eqn:di-approximate-control}
 \hat{u}^{(k)}(\tau) = \frac{2}{t_f}\frac{1}{\alpha_k}\dot{\hat{v}}^{(k)}(\tau), \quad (k=1,2) , 
\end{equation}
where $\hat{v}^{(k)}(\tau)$ is the $N=2$ Lagrange polynomial approximation of the state $v^{(k)}(\tau)$.  Figure~\ref{fig:VarMeshDoubleintegrator} shows the NLP control values $U_i^{(k)},\;(i,k=1,2)$ obtained from solving the NLP using $N=2$ LGR points in each of the two mesh intervals alongside the approximate control given in Eq.~\eqref{eqn:di-approximate-control}.  It is seen that the approximate control bears little resemblance to the known bang-bang structure of the optimal control.  Next, the NLP solver returns a value $T_1\approx -0.3$ which is in significant error from the known optimal value $T_1^*=0$.  Furthermore, in the second mesh interval the approximate control $U^{(2)}(\tau)$ exceeds the upper limit $u_{\max}=+1$ given in the continuous optimal control problem of Eq.~\eqref{eqn:example}.  Finally, the NLP objective is approximately $6.0$ which is less than the optimal objective $2\sqrt{x_0}\approx 6.32456$ of the continuous optimal control problem.  Consequently, including the variable $T_1$ as part of the two-interval formulation results in a misleading solution with regard to the control structure, the objective, and the value of $T_1$.  As a result, the Lavrentiev gap in the formulation of Eq.~\eqref{eqn:exampleMI} is positive which implies that the search space is too large.  The reason that the search space is too large is because in the discrete problem the control constraint is imposed only at the LGR collocation points $(\tau_1,\ldots,\tau_N)$.  As a result, at the final point $\tau_{N+1}=+1$ the approximate control given in Eq.~\eqref{eqn:di-approximate-control} can violate the constraint as shown in Fig.~\ref{fig:MLSearchSpaceCon}.  

\begin{figure}[ht]
\centering
\includegraphics[scale=0.45]{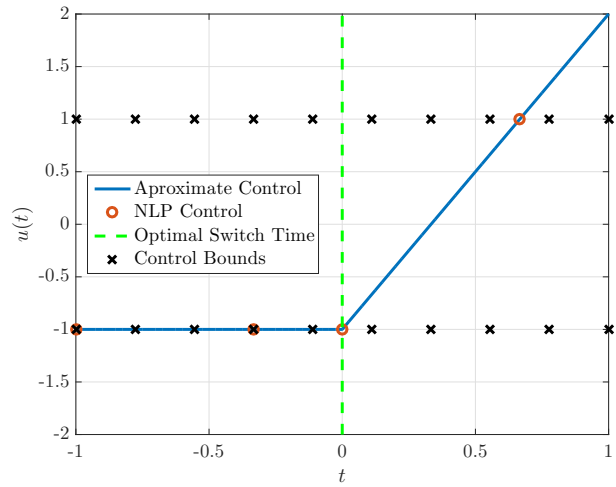}
\caption{Control obtained for two-interval formulation of example given in Eq.~\eqref{eqn:exampleMI}. \label{fig:VarMeshDoubleintegrator}}
\end{figure}

\begin{figure}[h!]
\centering 
\includegraphics[scale=0.45]{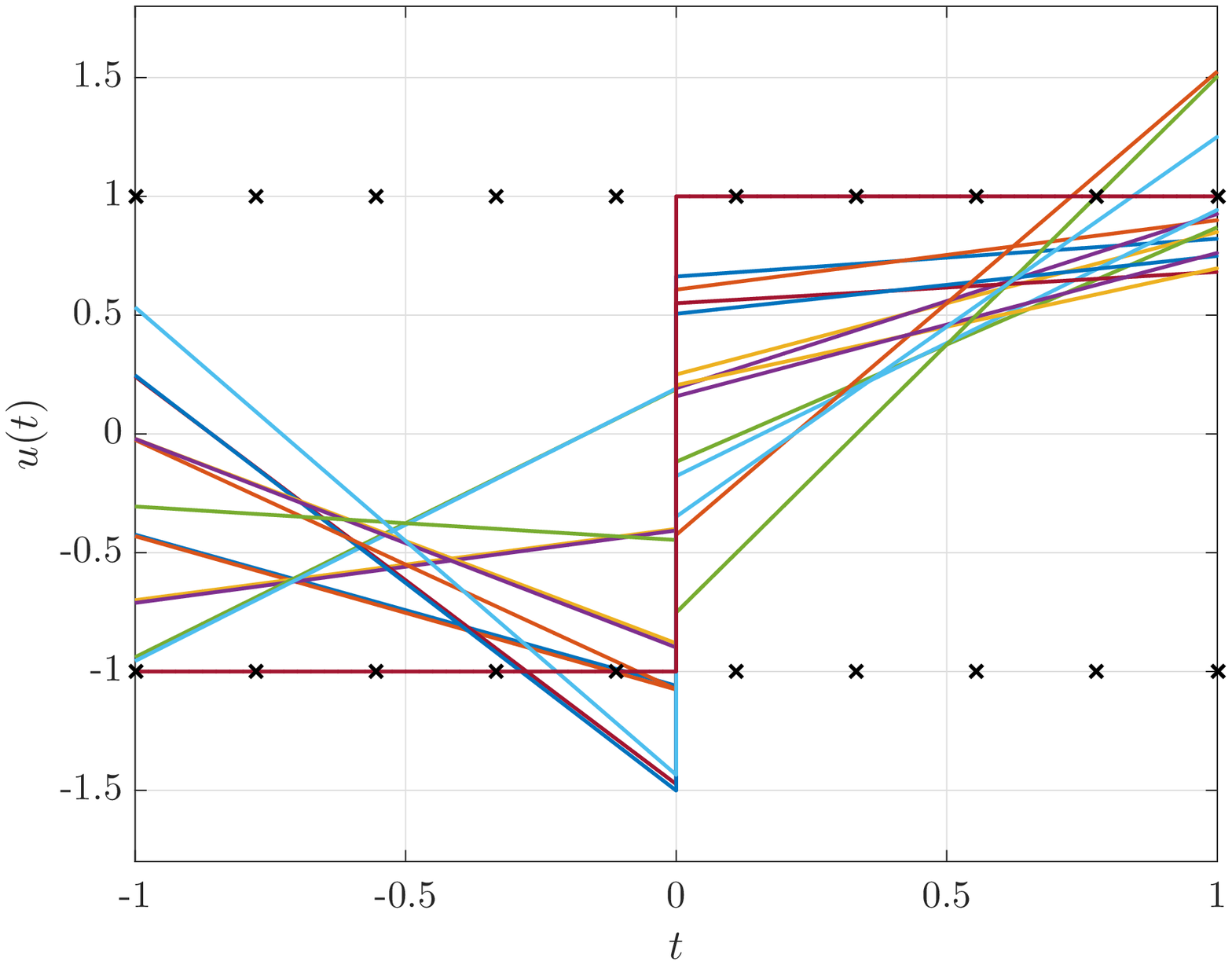}
\caption{Possible approximate control functions for the two-interval LGR approximation given in Eq.~\eqref{eqn:exampleMI} of the continuous optimal control problem given in Eq.~\eqref{eqn:example}.\label{fig:MLSearchSpaceCon}}
\end{figure}

% In order to modify the search space in the LGR collocation method, suppose now that the functions $(X\rk(\tau),V\rk(\tau))$ defined previously are restricted such that the only possible control functions $U\rk(\tau)$ are those such that the state and control approximations, $(X\rk(\tau),V\rk(\tau)$, and $U\rk(\tau))$, are feasible with respect to the control bounds $(u_{\min},u_{\max})$ given in Eq.~\eqref{eqn:exampleMI}.  In other words, the only possible state approximations are those that are feasible with respect to the control bounds and simultaneously satisfy all other constraints in the LGR NLP.  Using Eq.~\eqref{eqn:approximate-control-two-interval}, a search space different from that of the standard LGR collocation method can now be constructed that provides those control function approximations that lead to state approximations $(X\rk(\tau),V\rk(\tau))$ that are feasible with respect to the constraints of the continuous time optimal control problem.  

\section{Modified Legendre-Gauss-Radau Collocation\label{sect:LGRmod}}

Using the results of Section \ref{sect:Lavrentiev-LGR}, additional constraints are now augmented to the standard collocation method presented in Section \ref{sect:LGR} in order to improve the approximation of the location of the nonsmoothness in the solution to the optimal control problem (thereby improving the accuracy of the solution itself).  In particular collocation constraints are added at the end of each mesh interval, but such constraints are added to only those differential equations {\em that are a function of control}.  In this manner, and as stated in Section \ref{sect:introduction}, the approach developed in this research differs fundamentally from the approaches developed in Refs.~\cite{Cuthrell:1987} and \cite{Ross:2004}.  To simplify the following discussion, it will be assumed that each mesh interval for any given state discretization contains the same number of collocation points.  Therefore the differentiation matrix for any given state approximation is identical in each mesh interval.

\subsection{New Decision Variables\label{sect:new-variables}}

The modified LGR method introduces the following two new decision variables at the end of each mesh interval $\mathcal{I}_k,\;k=(1,\ldots,K)$.  The first new variables are the interior mesh points $T_{k},\;(k=1,\ldots,K-1)$.   The second new variable is the approximation of the control at the end of each mesh interval.  The value of this control approximation is denoted $\m{U}\rk_{N+1}$, $\left(k = 1,\hdots,K\right)$.  The portion of the decision vector associated with the control in the modified LGR collocation method is then defined as
\begin{displaymath}
  \tilde{\m{U}}\rk = \left[\begin{array}{c} \m{U}\rk \\ \m{U}\rk_{N+1} \end{array} \right].
\end{displaymath}
It is important to note that $\m{U}\rk_{N+1}$ and $\m{U}^{(k+1)}_{1}$ correspond to the same time point $T_k$.  In other words, $\m{U}\rk_{N+1}$ and $\m{U}^{(k+1)}_{1}$ correspond to the control at  $T_k^{-}$ and $T_k^{+}$, respectively.   This last point highlights the fact that the control need not be continuous at a mesh point.
% Now, it is convenient to make the following substitution: 
% \begin{equation*}
%  \alpha_k = \frac{T_{k} - T_{k-1}}{2}, \qquad \left(k = 1,\hdots, K -1\right).
% \end{equation*}
Reiterating, the two new variables in the modified LGR collocation method are the time at the end of each mesh interval, $T_k, \; \left(k = 1,\hdots, K-1\right)$, and the control at the end of each mesh interval, $\m{U}\rk_{N +1}, \; \left(k = 1,\hdots, K-1\right)$.

\subsection{New Constraints}

Given that variables have been added at the end of every mesh interval as described in Section \ref{sect:new-variables}, additional constraints must also be included in order to make the appropriate modifications to the search space.   In particular, collocation constraints are added at the end of each mesh interval using those differential equations that are a function of the control.  To understand why these new collocation constraints are included, consider the second differential equation $\dot{v}(\tau) =  \frac{t_f}{2}\alpha_k u(\tau)$ in the two-interval formulation of the example given in Eq.~\eqref{eqn:exampleMI} of Section \ref{sect:example-two-interval}.  Furthermore, suppose that $\hat{v}(\tau)$ is the Lagrange polynomial approximation of $v(\tau)$ and is a polynomial of degree $N$ in each of the two mesh intervals of the problem formulation given in Section \ref{sect:example}.  Finally, suppose that the constraint $\dot{v}(\tau) = \frac{t_f}{2}\alpha_k  u(\tau)$ is enforced at the $N$ LGR points plus the final point of every mesh interval.  Because $\hat{v}(\tau)$ is a polynomial of degree $N$ in each mesh interval and the differential equation depends upon the control, it is possible to satisfy the $N+1$ conditions
\begin{equation}\label{additional-collocation-double-integrator}
  \dot{\hat{v}}\rk(\tau_i) - \alpha_k \frac{t_f}{2} U_i\rk = 0, \quad (i=1,\ldots,N+1;\;  k=1,2)
\end{equation}
in each mesh interval because the control is a variable in Eq.~\eqref{additional-collocation-double-integrator}.  In other words, $U_{N+1}\rk$ can be varied in order to satisfy Eq.~\eqref{additional-collocation-double-integrator} at the endpoint of the first interval.  Moreover, when adding this collocation condition, it is also necessary to add the constraint that $u_{\min} \leq U_{N+1}\rk \leq u_{\max}$ in order to ensure that the control at the end of every mesh interval satisfies the limits on the control.

The preceding argument leads to a modification of the LGR collocation method for the case where the solution may be nonsmooth.  A collocation condition similar to that given in Eq.~\eqref{additional-collocation-double-integrator} is included along with a constraint that enforces all control bounds at the end of the mesh interval.  Adding a collocation condition at the end of a mesh interval results in a modified LGR differentiation matrix of the form
\begin{equation} \label{eqn:modD}
  \tilde{\m{D}} = 
  \left[
    \begin{array}{c}
      \m{D} \\
      \left[\begin{array}{ccc} \dot{\ell}_{1}(\tau_{N+1}),\ldots,\dot{\ell}_{N+1}(\tau_{N+1})\end{array}\right]
      % \m{E} 
    \end{array}
  \right] \in\mathbb{R}^{(N+1)\times(N+1)}, 
\end{equation}
where is is noted that $\tilde{\m{D}}$ is a matrix of size $(N+1)\times(N+1)$ and the last row of $\tilde{\m{D}}$ is given as
\begin{equation}\label{eqn:last-row-of-Dtilde}
  \tilde{\m{D}}_{(N+1,1:N+1)}= \left[\dot{\ell}_{1}(\tau_{N+1}),\ldots,\dot{\ell}_{N+1}(\tau_{N+1})\right]\in\mathbb{R}^{N+1}.
\end{equation}
% $ where $\m{E} \in \R^{1 \times N+1}$.
% It is important to note that the matrix $\tilde{\m{D}}\in\mathbb{R}^{(N+1)\times(N+1)}$
It is important to note that the matrix $\tilde{\m{D}}$ is used to collocate those differential equations that depend upon the control.  Furthermore, it is noted that the matrix $\m{D}$ in Eq.~\eqref{eqn:modD} is the standard LGR differentiation matrix as given in Section \ref{sect:LGR} \cite{Garg:2011a,Garg:2011b,Garg:2010}.  Including the new collocation constraint, Eq.~\eqref{eqn:LGRcon} is replaced with
\begin{equation}\label{eqn:modified-LGRcon}
  \begin{array}{lcl}
    \m{D}_{(i,:)}\m{X}\rk -\alpha_k\frac{t_f-t_0}{2}\m{f}_x\left(\m{X}\rk_i,\m{V}\rk_i\right) & = & \m{0}, \quad (i=1,\ldots,N),\\
    \tilde{\m{D}}_{(i,:)}\m{V}\rk - \alpha_k\frac{t_f-t_0}{2}\m{f}_v\left(\m{X}_i\rk,\m{V}_i\rk,\tilde{\m{U}}_{i}\rk\right) & = & \m{0}, \quad (i=1,\ldots,N+1).
%     \m{D}\m{X}\rk -\alpha_\frac{t_f-t_0}{2}\m{f}_x\left(\m{X}\rk_{1:N},\m{V}\rk_{1:N}\right) & = & \m{0}, \\
%     \tilde{\m{D}}\m{V}\rk - \alpha_k\frac{t_f-t_0}{2}\m{f}_v\left(\m{X}_{1:N+1}\rk,\m{V}_{1:N+1}\rk,\tilde{\m{U}}_{1:N+1}\rk\right) & = & \m{0}. 
  \end{array}
\end{equation}
Observe that, consistent with the explanation provided earlier in this section, the first constraint in Eq.~\eqref{eqn:modified-LGRcon} is {\em not} a function of control and, as a result, is identical to the first constraint given in Eq.~\eqref{eqn:LGRcon}. 

Additional constraints are added for the new $\alpha_k$ decision variable.  These additional constraints are
\begin{align}
\alpha_k &> 0,  \quad \left(k = 1,\hdots, K\right), \\
\sum^{K}_{k = 1}\alpha_k - 1 &= 0.\label{eqn:alphacon2}
\end{align}
These two constraints ensure that each element $\alpha_k,\;\left(k=1,\ldots,K\right),$ is always positive and that the sum is equal to unity.  The objective function given in Eq.~\eqref{eqn:LGRcost}, together with the constraints in Eq.~\eqref{eqn:modified-LGRcon}--\eqref{eqn:alphacon2}, is referred to as the {\em modified Legendre-Gauss-Radau collocation method}.

\subsection{Search Space of Modified LGR Method \label{sect:modSearchSpace}}

The example of Section \ref{sect:example} is now revisited using the modified LGR collocation method.  Figure \ref{fig:SSMod} exhibits the impact of the additional collocation constraint from Eq.~\eqref{eqn:modified-LGRcon} has on the search space of the example problem. 
\begin{figure}[ht!]
\centering
\subfloat[Admissible controls for two-interval formulation of example given in Eqs.~\eqref{eqn:exampleMI} using modified LGR collocation.\label{fig:doubleIntegratorSearchSpaceMod}]{\includegraphics[scale=0.45]{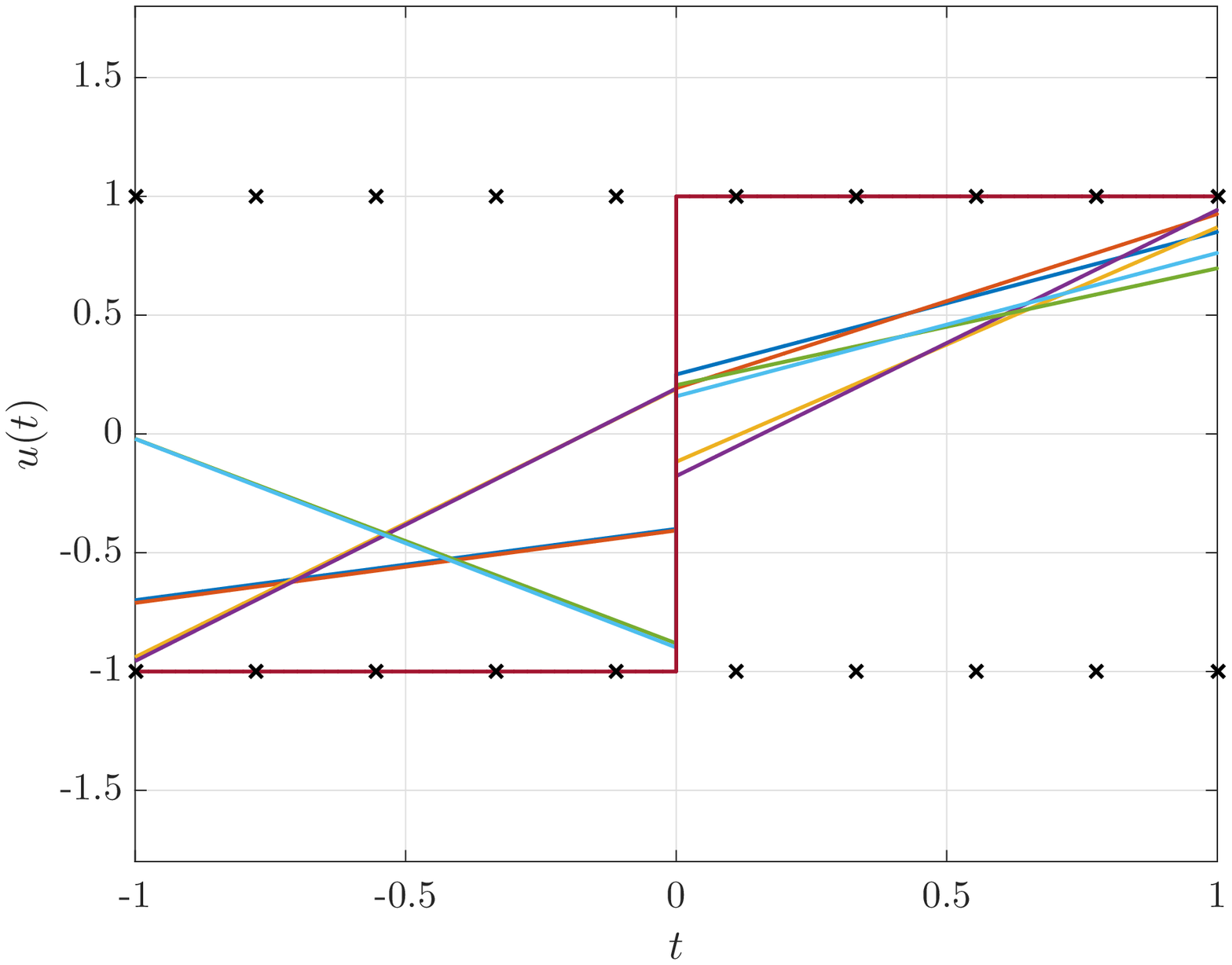}}~~~~\subfloat[Optimal objective, $J^*=t_f^*$, vs.~$T_1$ for two-interval formulation of example given in Eqs.~\eqref{eqn:exampleMI}.\label{Fig:CostVSwitchTime}]{\includegraphics[scale=0.45]{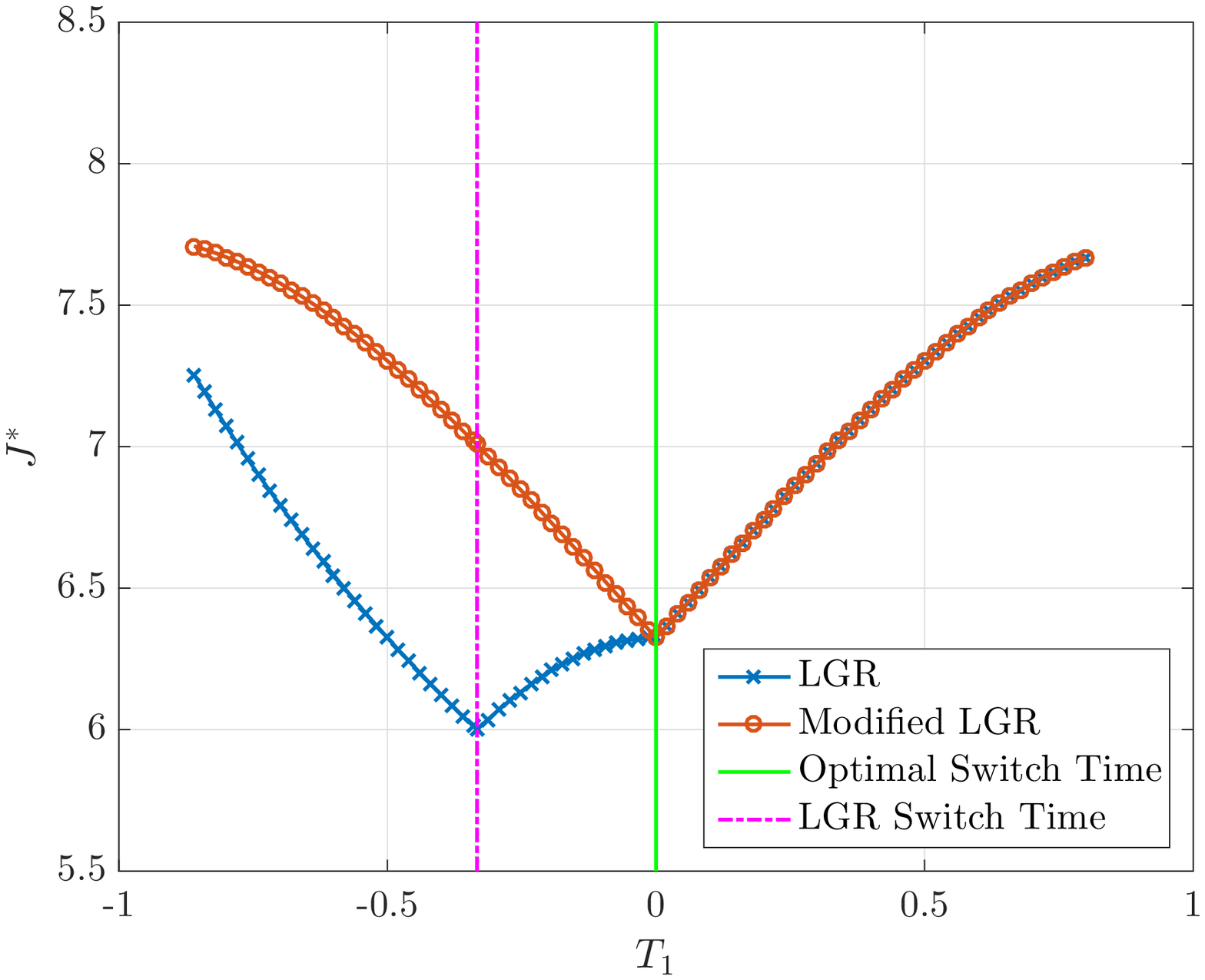}}
\caption{Admissible controls and optimal control for modified LGR collocation and comparison of optimal objective, $J^*$ vs.~switch time $T_1$ for standard and modified LGR collocation.}\label{fig:SSMod}
\end{figure}
Fig.~\ref{fig:doubleIntegratorSearchSpaceMod} demonstrates that each admissible set for control now falls between the allowable control limits $(u_{\min},u_{\max})=(-1,+1)$.  Next, to examine the effect that the modified LGR method has on the solution of the NLP for the example in Section \ref{sect:example}, Fig.~\ref{Fig:CostVSwitchTime} shows the objective of the modified LGR NLP as a function of the switch time, $T_1$, where it is assumed that the switch time is fixed.  At the optimal switch time $T_1^*$, the objective of both the original and modified LGR methods is identical.  Note, however, that when for $T_1 <T_1^*$, the optimal objective of the standard LGR method is {\em smaller} than the modified LGR method.  In fact, Fig.~\ref{Fig:CostVSwitchTime} shows that the optimal objective for the modified LGR method occurs when $T_1 <T_1^*$.  This last result indicates that the modified LGR method reduces the allowable search space such that the solution of the NLP leads to a state approximation that is closer to the solution of the continuous optimal control problem.  Figure~\ref{fig:doubleIntegratorVarMeshLav} shows the control solution obtained by solving for the control as a function of time using the Lagrange polynomial approximation of the state obtained using the modified LGR collocation method. 
\begin{figure}[ht!]
\centering
{\includegraphics[scale=0.45]{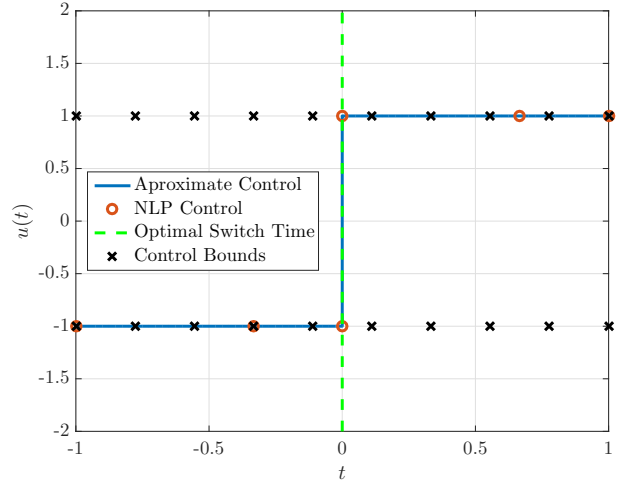}}
\caption{Optimal control for the example defined by Eq.~\eqref{eqn:example} using the modified LGR method.\label{fig:doubleIntegratorVarMeshLav}}
\end{figure}
 It is seen that, not only does the control function lie within its allowable limits $(u_{\min},u_{\max})=(-1,+1)$, but the switch time obtained using the modified LGR collocation method matches the switch time of the solution of the continuous optimal control problem.  

\section{Transformed Adjoint System and Weierstrass-Erdmann Conditions\label{sect:modadjoint}}

This section derives the adjoint system of the modified LGR collocation method based on the optimal control problem given in Eqs.~\eqref{eqn:contcost}, \eqref{eqn:contdyn}, and \eqref{eqn:contbound}.   In order to simplify the derivation, the state and control inequality path constraint given in Eq.~\eqref{eqn:contpath-mi} is dropped.  The first-order optimality conditions for the continuous optimal control problem are given as
\begin{align}\label{eqn:firstorderopt1}
\dot{\g{\lambda}}_x &= -\frac{\partial\mathcal{\mathcal{L}}}{\partial \m{x}} - \g{\lambda}_x\left[\frac{\partial \m{f}_x}{\partial \m{x}}\right]\tr -
\g{\lambda}_v\left[\frac{\partial \m{f}_v}{\partial \m{x}}\right]\tr, \\
\dot{\g{\lambda}}_v &= -\frac{\partial\mathcal{\mathcal{L}}}{\partial \m{v}} -\g{\lambda}_x \left[\frac{\partial \m{f}_x}{\partial \m{v}}\right]\tr - 
{\g{\lambda}_v}\left[\frac{\partial \m{f}_v}{\partial \m{v}}\right]\tr,\\
0&= \frac{\partial\mathcal{\mathcal{L}}}{\partial \m{u}} + {\g{\lambda}_v}\left[\frac{\partial \m{f}_v}{\partial \m{u}}\right]\tr,\label{eqn:firstorderU} \\
\g{\lambda}_x(-1) &=-\frac{\partial \mathcal{M}}{\partial \m{x}(-1)}+\g{\psi}\left[\frac{\partial \m{b}}{\partial \m{x}(-1)}\right]\tr, \label{eqn:firstorderx0}\\
\g{\lambda}_v(-1) &=-\frac{\partial \mathcal{M}}{\partial \m{v}(-1)}+\g{\psi}\left[\frac{\partial \m{b}}{\partial \m{v}(-1)}\right]\tr, \label{eqn:firstorderv0}\\
\g{\lambda}_x(+1) &=\frac{\partial \mathcal{M}}{\partial \m{x}(+1)}-\g{\psi}\left[\frac{\partial \m{b}}{\partial \m{x}(+1)}\right]\tr, \label{eqn:firstorderxf}\\
\g{\lambda}_v(+1) &=\frac{\partial \mathcal{M}}{\partial \m{v}(+1)}-\g{\psi}\left[\frac{\partial \m{b}}{\partial \m{v}(+1)}\right]\tr, \label{eqn:firstordervf}
\end{align}
where $\g{\lambda}_x(\tau)\in\R^{n_x}$ and $\g{\lambda}_v(\tau)\in\R^{n_v}$.   The goal of this section is to derive the first-order optimality conditions, also known as the Karush-Kuhn-Tucker (KKT) conditions, of the modified LGR collocation method. Then, using these first-order optimality conditions, a transformation is derived that relates the dual variables of the modified LGR collocation method to the costates of the continuous optimal control problem.  

\subsection{Derivation of Transformed Adjoint System\label{sect:tas-derivation}}

The derivation of the transformed adjoint system for the modified LGR collocation method proceeds as follows.  First, the Lagrangian associated with the modified LGR collocation constraints of Eq.~\eqref{eqn:modified-LGRcon} is given as
\begin{equation} \label{eqn:ModLGRAugCost}
\begin{aligned}
  J_a &= \mathcal{M}(\m{X}^{(1)}_1,\m{V}^{(1)}_1,\m{X}\rK_{N+1},\m{V}\rK_{N+1},t_0,t_f) + \sum_{k=1}^{K}\alpha_k\frac{t_f-t_0}{2}\sum_{i=1}^{N} w_i \mathcal{L}(\m{X}\rk_i,\m{V}\rk_i,\bmU\rk_i) \\
  & -\sum_{k=1}^{K}\sum_{i=1}^{N} \left\langle\g{\Lambda}_{x_i}\rk,\m{D}_{(i,1:N+1)}\m{X}\rk   - \alpha_k\frac{t_f - t_0}{2}\m{f}_{x}(\m{X}\rk_i,\m{V}\rk_i) \right\rangle \\
&- \sum_{k=1}^{K}\sum^{N+1}_{i = 1}\left\langle{\g{\Lambda}}_{v_i}\rk,\tilde{\m{D}}_{(i,1:N+1)}\m{V}\rk - \alpha_k\frac{t_f - t_0}{2} \m{f}_v(\m{X}\rk_i,\m{V}\rk_i,\m{U}\rk_i)\right\rangle\\
&-\g{\Psi} \m{b}\tr(\m{X}^{(1)}_1,\m{V}^{(1)}_1,\m{X}\rK_{N+1},\m{V}\rK_{N+1},t_0,t_f) - \beta\left(\sum^K_{k=1}\alpha_k-1\right),
\end{aligned}
\end{equation}
where $\bmL\rk_{x} \in \R^{N\times n_x}$, ${\bmL}_v\rk \in \R^{(N+1)\times n_v}$, $\g{\Psi} \in \R^{b}$, $\beta\in\R$ and $\left\langle \cdot, \cdot \right\rangle$ denotes the standard inner product between two vectors.  Furthermore, $\bmL\rk_{x_i}$ and ${\bmL}_{v_i}\rk$ denoted the $i^{th}$ rows of $\bmL\rk_{x}$ and ${\bmL}_v\rk$, respectively.  Rewriting Eq.~\eqref{eqn:ModLGRAugCost} so that the final row of the state matrix is separated from the first $N$ rows gives
\begin{equation}
\begin{aligned}\label{eqn:ModLGRAugCost2}
J_a &= \mathcal{M}(\m{X}^{(1)}_1,\m{V}^{(1)}_1,\m{X}\rK_{N+1},\m{V}\rK_{N+1},t_0,t_f) + \sum_{k=1}^{K}\alpha_k\frac{t_f-t_0}{2}\sum_{i=1}^{N} w_i \mathcal{L}(\m{X}\rk_i,\m{V}\rk_i,\bmU\rk_i) \\
&- \sum_{k=1}^{K}\sum_{i=1}^{N}\left\langle\g{\Lambda}\rk_{x_i},\m{D}_{(i,1:N+1)}\m{X}\rk_{1:N} + \m{D}_{(i,N+1)}\m{X}\rk_{N+1} - \alpha_k\frac{t_f-t_0}{2}\m{f}_x(\m{X}\rk_{i},\m{V}\rk_{i}) \right\rangle \\
&- \sum_{k=1}^{K}\sum_{i=1}^{N}\left\langle{\g{\Lambda}}\rk_{v_i},\m{D}_{(i,1:N)}\m{V}\rk_{1:N} + \m{D}_{(i,N+1)}\m{V}\rk_{N+1} - \alpha_k\frac{t_f-t_0}{2}\m{f}_v(\m{X}\rk_{i},\m{V}\rk_{i},\m{U}\rk_{i})\right\rangle \\
&- \sum_{k=1}^{K}\left\langle{\g{\Lambda}}\rk_{v_{N+1}},\tilde{\m{D}}_{(N+1,1:N)}\m{V}_{1:N}\rk +\tilde{\m{D}}_{(N+1,N+1)}\m{V}_{N+1}\rk\right\rangle\\
& + \sum_{k=1}^{K}\left\langle{\g{\Lambda}}\rk_{v_{N+1}},\alpha_k\frac{t_f-t_0}{2}\m{f}_v(\m{X}\rk_{N+1},\m{V}\rk_{N+1},\m{U}\rk_{N+1})\right\rangle\\
&-\g{\Psi} \m{b}\tr(\m{X}^{(1)}_1,\m{V}^{(1)}_1,\m{X}\rK_{N+1},\m{V}\rK_{N+1},t_0,t_f)\\
&-\beta\left(\sum^K_{k=1}\alpha_k-1\right),
\end{aligned}
\end{equation}
Next, the following theorem is introduced that will allow the terms involving $\m{f}_v(\m{X}\rk_{N+1}$, $\m{V}\rk_{N+1}$, $\m{U}\rk_{N+1})$, and $\tilde{\m{D}}_{(N+1,:)}$ in Eq.~\eqref{eqn:ModLGRAugCost2} to be written as functions of $\m{X}\rk_{1:N}$, $\m{V}\rk_{1:N}$, $\m{U}\rk_{1:N}$, and $\m{D}_{(:,N+1)}$.
\begin{theorem} \label{th:intfldot}
  Let $f(\tau)$ be a polynomial of degree at most $N-1$ on the interval $\tau\in[-1, 1]$.  Furthermore, let $(\tau_1,\ldots,\tau_N)$ be the Legendre-Gauss-Radau points on the interval $[-1, 1)$ and let $\tau_{N+1} = +1$.  Then, if $\ell_j(\tau)$ are the Lagrange polynomials given in Eq.~\eqref{eqn:lagrange-basis-lgr}, it is the case that
\begin{equation} \label{eqn:Dnp1theorem}
  \int^{+1}_{-1} f(\tau)\dot{\ell}_{N+1}(\tau) d\tau = f(+1).
\end{equation}
\end{theorem}
\begin{proof}
  From Eq.~\eqref{eqn:lagrange-basis-lgr}, the Lagrange polynomial $\ell_{N+1}(\tau)$ is given as
  \begin{equation}
  \ell_{N+1}(\tau) = \prod_{l=1}^{N} \frac{\tau - \tau_l}{\tau_{N+1} - \tau_l}. 
\end{equation}
Then the left-hand side of Eq.~\eqref{eqn:Dnp1theorem} can be integrated by parts as
\begin{equation} \label{eqn:DnpProof1}
 \int^{+1}_{-1} f(\tau)\dot{\ell}_{N+1}(\tau) d\tau  = f(\tau)\ell_{N+1}(\tau) \Big|^{+1}_{-1} - \int^{+1}_{-1} \dot{f}(\tau) \ell_{N+1}(\tau)d\tau.
\end{equation}
Because $f(\tau)$ is a polynomial of degree at most $N-1$, it follows that $\dot{f}(\tau)$ is a polynomial of degree at most $N-2$.  Furthermore, because $\ell_{N+1}(\tau)$ is a polynomial of at most degree $N$, then the integrand in Eq.~\eqref{eqn:DnpProof1}  is at most degree $2N-2$ and the integral can be evaluated exactly using LGR quadrature as
\begin{equation} \label{eqn:DnpProof2}
  \intoo \dot{f}(\tau) \ell_{N+1}(\tau) d\tau = \sum^{N}_{i = 1} w_i \dot{f}(\tau_i) \ell_{N+1}(\tau_i),  
\end{equation} 
where $w_i$ is the $i^{th}$ LGR quadrature weight.  Then, from Eq.~\eqref{convention-lagrange-basis-isolation-property}, every term $\ell_{N+1}(\tau_i)=0,\; (i\neq N+1)$ is zero which implies that
\begin{equation} \label{eqn:DnpProof3}
  \int^{+1}_{-1} f(\tau)\dot{\ell}_{N+1}(\tau) d\tau  = f(\tau)\ell_{N+1}(\tau) \Big|^{+1}_{-1}
\end{equation}
Consequently, Eq.~\eqref{eqn:DnpProof1} reduces to
\begin{equation}
  \int^{+1}_{-1} f(\tau)\dot{\ell}_{N+1}(\tau) d\tau = f(\tau)\ell_{N+1}(\tau) \Big|^{+1}_{-1}  = f(+1)\ell_{N+1}(+1) - f(-1)\ell_{N+1}(-1) = f(+1).
\end{equation}
\end{proof}
{\noindent}The result of Theorem \ref{th:intfldot} enables expressing the elements $\tilde{\m{D}}_{(N+1,j)},\; (j=1,\ldots,N)$ in Eq.~\eqref{eqn:modD} in terms of $\m{D}_{(:,N+1)}$ and $\m{D}_{(:,1:N)}$.  First, the $N$ elements of $\tilde{\m{D}}_{(N+1,j)},\; (j=1,\ldots,N)$ are defined as
\begin{equation} \label{eqn:Li}
  \tilde{\m{D}}_{(N+1,j)} = \dot{\ell}_j(+1), \quad \left(\jl\right).
\end{equation}
Then, replacing $f(\tau)$ in Eq.~\eqref{eqn:Dnp1theorem} with $\dot{\ell}_j(\tau),\;(j=1,\ldots,N)$, the quantities $\tilde{\m{D}}_{(N+1,j)},\; (j=1,\ldots,N)$ are given as
\begin{equation} \label{eqn:LiLf}
  \tilde{\m{D}}_{(N+1,j)} = \intoo \dot{\ell}_j(\tau)\dot{\ell}_{N+1}(\tau)d\tau, \quad \left(j = 1,\hdots, N\right).
\end{equation}
Because $\dot{\ell}_j(\tau)\dot{\ell}_{N+1}(\tau)$ is a polynomial of degree most $2N-2$, Eq.~\eqref{eqn:LiLf} can be replaced exactly with an LGR quadrature as
\begin{equation} \label{eqn:LiLf2}
  \tilde{\m{D}}_{(N+1,j)} = \sum^{N}_{i=1}w_i\dot{\ell}_j(\tau_i)\dot{\ell}_{N+1}(\tau_i), \quad \left(j = 1,\hdots, N\right).
\end{equation}
Noting that $\m{D}_{(i,N+1)}=\dot{\ell}_{N+1}(\tau_i)$ and that $\m{D}_{(i,j)} = \dot{\ell}_j(\tau_i)$, Eq.~\eqref{eqn:LiLf2} can be written as
\begin{equation} \label{eqn:LiLf3}
  \tilde{\m{D}}_{(N+1,j)} = \sum^{N}_{i=1}w_i\m{D}_{(i,j)}\m{D}_{(i,N+1)} = \m{D}_{(:,N+1)}\tr \m{W}\m{D}_{(:,j)}, \quad (j=1,\ldots,N),
\end{equation}
where $\m{W}=\textrm{diag}(w_1,\ldots,w_N)$ is a diagonal matrix of LGR quadrature weights $(w_1,\ldots,w_N)$.  The $N$ quantities $\tilde{\m{D}}_{(N+1,1:N)}$ given in Eq.~\eqref{eqn:LiLf3} can be written in a single equation as 
\begin{equation}\label{eqn:LiLf4}
  \tilde{\m{D}}_{(N+1,1:N)} = \m{D}_{(:,N+1)}\tr\m{W}\m{D}_{(:,1:N)}. 
\end{equation}

% ---------------------------------------
Now suppose that $(\m{X}_i\rk,\m{V}_i\rk,\m{U}_i\rk),\; (i=1,\ldots,N+1)$ satisfy the constraints given in Eq.~\eqref{eqn:modified-LGRcon}.  Then, in the case $i=N+1$
\begin{equation}\label{eqn:modified-lgr-collocation-on-v-at-tau=+1}
  \alpha_k \frac{t_f-t_0}{2}\m{f}_v\left(\m{X}_{N+1}\rk,\m{V}_{N+1}\rk,\m{U}_{N+1}\rk\right) = \sum_{j=1}^{N+1} \dot{\ell}_j(\tau_{N+1})\m{V}_j\rk,
\end{equation}
where $\tau_{N+1}=+1$.  Now, let $f(\tau)$ in Theorem \ref{th:intfldot} be chosen as the vector function
\begin{equation}\label{f(tau)-theorem-1}
  \m{F}(\tau) = \sum_{j=1}^{N+1}\dot{\ell}_j(\tau)\m{V}_J\rk, 
\end{equation}
where it is noted in Eq.~\eqref{f(tau)-theorem-1} that $\m{F}(\tau)$ is a polynomial of degree at most $N-1$.  Then, the result Eq.~\eqref{eqn:DnpProof2} gives
\begin{equation}\label{f(+1)-theorem-1}
  \m{F}(+1) =  \int_{-1}^{+1} \dot{\ell}_{N+1}(\tau) \m{F}(\tau) d\tau = \sum_{j=1}^{N+1} \ell_j(+1)\m{V}_j\rk.
\end{equation}
Next, LGR quadrature is exact for a polynomial of degree at most $2N-2$ and because $(\m{X}_i\rk,\m{V}_i\rk,\m{U}_i\rk),\; (i=1,\ldots,N+1)$ satisfy the constraints in  Eq.~\eqref{eqn:modified-LGRcon}.  Therefore, the integral in Eq.~\eqref{f(+1)-theorem-1} can be replaced with
\begin{equation}\label{integral-of-fdot-replaced-by-LGR-quadrature}
  \int_{-1}^{+1} \dot{\ell}_{N+1}(\tau) \m{F}(\tau) d\tau = \sum_{i=1}^{N} w_i \dot{\ell}_{N+1}(\tau_i) \m{F}(\tau_i),
\end{equation}
which implies that
\begin{equation}\label{eqn:modified-lgr-collocation-on-v-at-tau=+1-rewritten}
  \sum_{j=1}^{N+1} \ell_j(+1)\m{V}_j\rk = \sum_{i=1}^{N} w_i \dot{\ell}_{N+1}(\tau_i) \m{F}(\tau_i) = \alpha_k \frac{t_f-t_0}{2}\sum_{i=1}^{N} w_i \dot{\ell}_{N+1}(\tau_i)\m{f}_v\left(\m{X}_{i}\rk,\m{V}_{i}\rk,\m{U}_{i}\rk\right).
\end{equation}
Combining Eqs.~\eqref{eqn:modified-lgr-collocation-on-v-at-tau=+1} and \eqref{eqn:modified-lgr-collocation-on-v-at-tau=+1-rewritten} gives
\begin{equation}\label{eqn:DWDf}
  \m{f}_v \left(\m{X}_{N+1}\rk,\m{V}_{N+1}\rk,\m{U}_{N+1}\rk\right) = \m{D}_{(:,N+1)}\tr\m{W}\m{f}_v \left(\m{X}_{1:N}\rk,\m{V}_{1:N}\rk,\m{U}_{1:N}\rk\right). 
\end{equation}
% Equation~\eqref{eqn:LiLf4} can now be used to rewrite the augmented objective function in a more convenient manner.  Using the result in Eq.~\eqref{eqn:LiLf4}, 
% \begin{equation} \label{eqn:DWDf}
%   \m{D}_{(:,N+1)}\tr{\m{W}}\m{D}_{(:,1:N)}\m{V}_{1:N} = \tilde{\m{D}}_{(N+1,1:N)} \m{f}_v(\m{X}\rk_{1:N},\m{V}\rk_{1:N},\m{U}\rk_{1:N}) = \m{f}_{v}(\m{X}_{N+1},\m{V}_{N+1}, \m{U}_{N+1}). 
% \end{equation}
where
\begin{displaymath}
  \m{f}_v(\m{X}\rk_{1:N},\m{V}\rk_{1:N},\m{U}\rk_{1:N}) \equiv \left[\begin{array}{c}\m{f}_v(\m{X}\rk_{1},\m{V}\rk_{1},\m{U}\rk_{1}) \\ \vdots \\ \m{f}_v(\m{X}\rk_{N},\m{V}\rk_{N},\m{U}\rk_{N}) \end{array}\right].
\end{displaymath}
Then, subsituting the identities given in Eqs.~\eqref{eqn:LiLf4} and \eqref{eqn:DWDf} into the Lagrangian of Eq.~\eqref{eqn:ModLGRAugCost2}  gives
% Substituting the result of Eq.~\eqref{eqn:DWDf} into Eq.~\eqref{eqn:ModLGRAugCost2} gives
%
\begin{equation}
\begin{aligned}
J_a &= \mathcal{M}(\m{X}^{(1)}_1,\m{V}^{(1)}_1,\m{X}\rK_{N+1},\m{V}\rK_{N+1},t_0,t_f) + \sum^K_{k=1}\alpha_k\frac{t_f-t_0}{2}\sum_{i=1}^{N} w_i \mathcal{L}(\m{X}\rk_i,\m{V}\rk_i,\bmU\rk_i) \\
&- \sum^K_{k=1}\sum_{i=1}^{N}\left\langle\g{\Lambda}\rk_{x_i},\m{D}_{(i,1:N)}\m{X}\rk_{1:N}+\m{D}_{(i,N+1)}\m{X}\rk_{N+1} - \alpha_i\rk\frac{t_f - t_0}{2}\m{f}_{x}(\m{X}\rk_{i},\m{V}\rk_{i})\right\rangle \\
&- \sum^K_{k=1}\sum_{i=1}^{N}\left\langle{\g{\Lambda}}\rk_{v_i},\m{D}_{(i,1:N)}\m{V}\rk_{1:N} - \alpha_k\frac{t_f - t_0}{2}\m{f}_v(\m{X}\rk_{i},\m{V}\rk_{i},\m{U}\rk_{i})\right\rangle \\
&- \sum^K_{k=1}\sum_{i=1}^{N}\left\langle{\g{\Lambda}}\rk_{v_i},\m{D}_{(i,N+1)}\m{f}\rk_{v_{N+1}}\right\rangle \\
&- \sum^K_{k=1}\left\langle{\g{\Lambda}}\rk_{v_{N+1}},\m{D}_{(:,N+1)}\tr{\m{W}}\m{D}_{(1:N,:)}\m{V}\rk_{1:N} +\tilde{\m{D}}_{(N+1,N+1)}\m{V}\rk_{N+1}\right\rangle \\
&+ \sum^K_{k=1}\left\langle{\g{\Lambda}}_{v_{N+1}},\alpha_i\rk\frac{t_f - t_0}{2}\left(\m{D}_{(:,N+1)}\tr\m{W}\m{f}_v(\m{X}\rk_{1:N},\m{V}\rk_{1:N},\m{U}\rk_{1:N})\right)\right\rangle\\
&-\g{\Psi} \m{b}\tr(\m{X}^{(1)}_1,\m{V}^{(1)}_1,\m{X}\rK_{N+1},\m{V}\rK_{N+1},t_0,t_f)\\
&-\beta\left(\sum^K_{k=1}\alpha_k-1\right). \\
\end{aligned}
\end{equation}

Now, to simplify the derivations that follow, the following substitutions will be made:
\begin{equation}
  \begin{array}{lclclclclcl}
    \mathcal{L}_i\rk = \mathcal{L}(\m{X}\rk_i,\m{V}\rk_i,\m{U}\rk_i), \\
    \m{f}_{x_i}\rk = \m{f}_x(\m{X}\rk_i,\m{V}\rk_i), \\
    \m{f}_{v_i}\rk = \m{f}_v(\m{X}\rk_i,\m{V}\rk_i,\m{U}\rk_i).
  \end{array}
\end{equation}
The KKT conditions are then derived by taking the partial derivatives $J_a$ with respect to $\m{X}\rk$, $\m{V}\rk$, $\bmU\rk$, $\g{\Lambda}_x\rk$,${\g{\Lambda}}_v\rk$, $\g{\Psi}\rk$, $t_0$, $t_f$ and $\alpha_k$ and setting them equal to zero.  These derivatives are given as follows:
\begin{eqnarray} \label{eqn:LGRconAdj1}
  \m{D}_{(i,:)}\m{X}\rk -\alpha_k\frac{t_f-t_0}{2}\m{f}_x\left(\m{X}\rk_i,\m{V}\rk_i\right) & = & \m{0}, \quad (i=1,\ldots,N),\\
  \tilde{\m{D}}_{(i,:)}\m{V}\rk - \alpha\rk\frac{t_f-t_0}{2}\m{f}_v\left(\m{X}_i\rk,\m{V}_i\rk,\tilde{\m{U}}_{i}\rk\right) & = & \m{0}, \quad (k=1,\ldots,N+1), \label{eqn:LGRconAdj2}\\
  \m{b}(\m{X}^{(1)}_1,\m{V}^{(1)}_1,\m{X}\rK_{N+1},\m{V}\rK_{N+1},t_0,t_f)&= & \m{0}, \label{eqn:LGRBoundAdj}\\
  \sum_{k=1}^K \alpha_k -1 &= & 0 \label{eqn:alphcon},
\end{eqnarray}
\begin{equation}\label{eqn:modLGRKKTX}
\begin{split}
\m{D}_{(i,:)}\tr\g{\Lambda}\rk_x &=\alpha_k\frac{t_f - t_0}{2} \nabla_{\m{X}}\left(w_i \mathcal{L}\rk_i +  \left\langle {\g{\Lambda}}\rk_{x_i}  ,\m{f}\rk_{x_i}\right\rangle+\left\langle {\g{\Lambda}}\rk_{v_i} + {\g{\Lambda}}\rk_{v_{N+1}} \m{D}_{(i,N+1)} w_i ,\m{f}\rk_{v_i}\right\rangle\right)\\
&-\delta_{1i} (-\nabla_{\m{X}}\mathcal{M}+\nabla_{\m{X}}\g{\Psi}\m{b}\tr),
\end{split}
\end{equation}
\begin{equation}\label{eqn:modLGRKKTXnp1}
\begin{split}
\m{D}_{(:,N+1)}\tr\g{\Lambda}_x &= \nabla_{\m{X}}\mathcal{M}-\nabla_{\m{X}}\g{\Psi}\m{b}\tr,
\end{split}
\end{equation}
\begin{equation}\label{eqn:modLGRKKTV}
\begin{split}
  \m{D}_{(i,:)}\tr\left({\g{\Lambda}}_{v_{1:N}}\rk + {\g{\Lambda}}\rk_{v_{N+1}} \m{D}_{(i,N+1)} w_i\right)&=\alpha_k\frac{t_f - t_0}{2}\nabla_{\m{V}}\left(w_i \mathcal{L}\rk_i+\left\langle {\g{\Lambda}}\rk_{x_i}  ,\m{f}\rk_{x_i}\right\rangle\right) \\
  & +\alpha_k\frac{t_f - t_0}{2}\nabla_{\m{V}}\left(\left\langle{\g{\Lambda}}_{v_i}\rk + {\g{\Lambda}}_{v_{N+1}}\rk \m{D}_{(i,N+1)} w_i ,\m{f}\rk_{v_i}\right\rangle\right)\\
&+\g{\Lambda}\rk_i-{\delta}_{1i}(-\nabla_{\m{V}}\mathcal{M} + \nabla_{\m{V}}\g{\Psi}\m{b}\tr),
\end{split}
\end{equation}
\begin{equation}\label{eqn:modLGRKKTVnp1}
\begin{split}
\m{D}_{(:,N+1)}\tr{\g\Lambda}_{v_{1:N}}+\tilde{\m{D}}_{(N+1,N+1)}{\g{\Lambda}}_{v_{N+1}}\rK &= \nabla_{\m{V}}\mathcal{M}-\nabla_{\m{V}}\g{\Psi}\m{b}\tr,
\end{split}
\end{equation}
\begin{equation}\label{eqn:modLGRKKTU}
\begin{split}
\m{0} &= \alpha_k\frac{t_f - t_0}{2}\nabla_{\m{U}}\left(w_i\mathcal{L}\rk_i - \left\langle{\g{\Lambda}}_{v_{1:N}}\rk +{\g{\Lambda}}\rk_{v_{N+1}} \m{D}_{(i,N+1)}w_i,\m{f}\rk_{v_{1:N}}\right\rangle\right),\\
& (k= 1,\hdots,K;\;i = 1,\hdots,N),
\end{split}
\end{equation}
\begin{equation}\label{eqn:modLGRKKTt0}
\begin{split}
  0 &=\sum^K_{k = 1}\frac{-\alpha_k}{2}\sum_{i=1}^{N} w_i \mathcal{L}_i\rk + \sum_{i=1}^{N}\left\langle\g{\Lambda}\rk_{x_i}, \frac{-\alpha_k}{2}\m{f}\rk_{x_i}\right\rangle + \sum_{i=1}^{N}\left\langle{\g{\Lambda}}\rk_{v_i},\frac{-\alpha_k}{2}\m{f}_{v_i}\rk\right\rangle \\
  & + \left\langle{\g{\Lambda}}\rk_{v_{N+1}},\frac{-\alpha_k}{2}\left(\m{D}\tr_{N+1}\m{W}\m{f}\rk\right)\right\rangle+\nabla_{t_0}\left(\mathcal{M} - \g{\Psi}\m{b}\tr\right),
\end{split}
\end{equation}
\begin{equation}\label{eqn:modLGRKKTtf}
\begin{split}
  0 &=\sum_{k=1}^K\frac{\alpha_k}{2}\sum_{i=1}^{N} w_i \mathcal{L}\rk_i + \sum_{i=1}^{N}\left\langle\g{\Lambda}\rk_{x_i}, \frac{\alpha_k}{2}\m{f}\rk_{x_i}\right\rangle + \sum_{i=1}^{N}\left\langle{\g{\Lambda}}\rk_{v_i},\frac{\alpha_k}{2}\m{f}_{v_i}\rk\right\rangle \\
  & + \left\langle{\g{\Lambda}}\rk_{v_{N+1}},\frac{\alpha_k}{2}\left(\m{D}\tr_{N+1}\m{W}\m{f}_v\rk\right)\right\rangle+\nabla_{t_f}\left(\mathcal{M} - \g{\Psi}\m{b}\tr\right),
\end{split}
\end{equation}
\begin{equation}\label{eqn:modLGRKKTa}
\begin{split}
  0 &=\frac{t_f - t_0}{2}\sum_{i=1}^{N} w_i \mathcal{L}\rk_i + \sum_{i=1}^{N}\left\langle\g{\Lambda}\rk_{x_i}, \frac{t_f - t_0}{2}\m{f}\rk_{x_i}\right\rangle + \sum_{i=1}^{N}\left\langle{\g{\Lambda}}\rk_{v_i},\frac{t_f - t_0}{2}\m{f}_{v_i}\rk\right\rangle \\
  & + \left\langle{\g{\Lambda}}\rk_{v_{N+1}},\frac{t_f - t_0}{2}\left(\m{D}\tr_{N+1}\m{W}\m{f}_v\rk\right)\right\rangle - \beta \qquad (k = 1,\hdots,K),
\end{split}
\end{equation}
where $\delta_{ij}$ is the Kronecker delta function defined as
\begin{equation}
\delta_{ij} = \begin{cases}
1, \quad i = j \\
0, \quad i\neq j.
\end{cases}
\end{equation}  
The KKT conditions given in equation \eqref{eqn:modLGRKKTa} are unique to the modified LGR method and is not required for an extremal solution of the standard LGR NLP transcription. Now propose the change of variables
\begin{align}
\label{eqn:modLGRSubX}
{\g{\lambda}}\rk_{x_i} &= \frac{\g{\Lambda}_{x_{i}}\rk}{w_i}, \\
\label{eqn:modLGRSubXnp1}
\g{\lambda}\rk_{x_{N+1}} &= \m{D}_{(:,N+1)}\tr\g{\Lambda}\rK_{x_{1:N}},\\ 
\label{eqn:modLGRSubPsi}
\g{\psi}_i &= \g{\Psi}_i,\\
\label{eqn:modLGRSubV}
{\g{\lambda}}\rk_{v_i} &= \frac{{\g{\Lambda}}_{v_i}\rk}{w_i} + {\g{\Lambda}}\rk_{v_{N+1}}\m{D}_{(i,N+1)},\\
\label{eqn:modLGRSubVnp1}
{\g{\lambda}}_{v_{N+1}}\rK &= \m{D}_{(:,N+1)}\tr{\g{\Lambda}}\rK_{v_{1:N}}+{\g{\Lambda}}_{v_{N+1}} \tilde{\m{D}}_{(N+1,N+1)}.
\end{align}
Note that Eqs.~\eqref{eqn:modLGRSubX}--\eqref{eqn:modLGRSubPsi} are the same transformations used for the standard LGR method.  Finally, define $\m{D}^\dagger \in R^{N\times N}$ such that
\begin{align}
\label{eqn:Ddagger}
\m{D}_{(1,1)}^{\dagger} &= -\m{D}_{(1,1)} - \frac{1}{w_1} \\
\m{D}_{(i,j)}^{\dagger} &= -\frac{w_j}{w_i}\m{D}_{(j,i)} \quad \text{otherwise},
\end{align}
for $i = j = 1,2, \hdots,N$. Note that $\m{D}^{\dagger}$ is the same matrix derived by Refs.~\cite{Garg:2011a,Garg:2011b} where it was shown that $\m{D}^{\dagger}$ is the differentiation matrix for the space of polynomials of degree at most $N-1$.  Now the KKT conditions can be rewritten as
\begin{align}
\begin{split}
\label{eqn:DdLam1}
\m{D}^\dagger_i\g{\lambda}\rk_{x_{1:N}} &= -\alpha_k\frac{t_f - t_0}{2}\nabla_{\m{X}}\left(\left\langle{\g{\lambda}}\rk_{x_i},\m{f}_{x_i}\rk\right\rangle+\left\langle{\g{\lambda}}\rk_{v_i},\m{f}_{v_i}\rk\right\rangle+\mathcal{L}_i\rk\right)\\
&+ \frac{{\delta}_{1i}}{w^{(1)}_1}\left(-\nabla_{X}\left(\mathcal{M} - \g{\psi}\m{b}\tr\right) - \g{\lambda}^{(1)}_{x_1} \right),
\end{split}\\
\begin{split}
\label{eqn:DdLam2}
\m{D}^\dagger_i{\g{\lambda}}_{v_{1:N}}\rk &= -{\alpha_k}\frac{t_f - t_0}{2}\nabla_{\m{V}}\left(\left\langle{\g{\lambda}}\rk_{x_i},\m{f}_{x_i}\rk\right\rangle+\left\langle{\g{\lambda}}\rk_{v_i},\m{f}_{v_i}\rk\right\rangle+\mathcal{L}_i\rk\right) \\
&+ \frac{{\delta}_{1i}}{w_1}\left(-\nabla_{\m{V}}\left(\mathcal{M} - \g{\psi}\m{b}\tr\right) - {\g{\lambda}}^{(1)}_{v_1} \right),
\end{split}\\
\begin{split}
\label{eqn:JaDU}
\m{0} &= \alpha_k\frac{t_f - t_0}{2}\nabla_{\m{U}}\left(\mathcal{L}_i\rk - \left\langle {\g{\lambda}}\rk_{v_i},\m{f}_{v_i}\rk\right\rangle\right),\\
&(i = 1,\hdots,N, \quad k = 1,\hdots, K),
\end{split}\\
\begin{split}
\label{eqn:lamxf}
\g{\lambda}\rK_{x_{N+1}} &=\nabla_{\m{X}}\left(\mathcal{M} -\g{\psi}\m{b}\tr\right),
\end{split}\\
\begin{split}
\label{eqn:lamvf}
{\g{\lambda}}\rK_{v_{N+1}} &=\nabla_{\m{V}}\left(\mathcal{M} -\g{\psi}\m{b}\tr\right),
\end{split}\\
\begin{split}
\label{eqn:hamt0}
-\nabla_{t_0}\left(\mathcal{M} - \g{\Psi}\m{b}\tr\right) = \sum^{K}_{k=1}-\alpha_k\sum^{N}_{i = 1}{H}_i\rk w_i,
\end{split}\\
\begin{split}
\label{eqn:hamtf}
-\nabla_{t_f}\left(\mathcal{M} - \g{\Psi}\m{b}\tr\right) = \sum^{K}_{k=1}\alpha_k\sum^{N}_{i = 1}{H}_i\rk w_i,
\end{split}
\end{align}
where $H_i\rk = \mathcal{L}\rk_i + \bml\rk_{x_i} \m{f}\rktr_{x_i} + {\bml}\rk_{v_i} \m{f}\rktr_{v_i}$ is the approximation the Hamiltonian, $\mathcal{H}$ in interval $k$. Equations~\eqref{eqn:firstorderx0}-\eqref{eqn:firstorderv0} allow the terms in the second lines of Eqs.~\eqref{eqn:DdLam1}--\eqref{eqn:DdLam2} to vanish which results in Eqs.~\eqref{eqn:DdLam1}--\eqref{eqn:hamtf} becoming discrete representations of the continuous time first-order optimality conditions from Eqs.~\eqref{eqn:firstorderopt1}--\eqref{eqn:firstordervf}.

\subsection{Weierstrass-Erdmann Conditions\label{sect:WECond}}

If the optimal control is discontinuous, additional optimality conditions called the {\em Weierstrass-Erdmann conditions} \cite{br:aoc} must be satisfied.  One of the Weierstrass-Erdmann conditions states that the Hamiltonian must be continuous at the location of a control discontinuity.  The Hamiltonian for the optimal control problem defined in Eqs.~\eqref{eqn:exampleMI} can be approximated as
\begin{equation}
\label{eqn:modHam}
\mathcal{H}\rk(\tau_i\rk) \approx H\rk_i = \mathcal{L}\rk_i + \bml\rk_{x_ i} \m{V}\rktr_i + \bml\rk_{v_i} \m{f}\rktr_i,
\end{equation}
where $\tau\rk_i \in [T_{k-1}, T_k]$, $i = 1,\hdots,N$ and $k = 1,\hdots,K$ are the $N$ LGR points in the $k^{th}$ mesh interval. The Weierstrass-Erdmann condition on the Hamiltonian can be written as \cite{br:aoc} 
\begin{equation}
\label{eqn:WECond}
\mathcal{H}(T_1^-) = \mathcal{H}(T_1^+),
\end{equation}
where $T_1^-$ and $T_+1$ on the left-hand side and right-hand side of the discontinuity, respectively.  

The analysis that follows will demonstrate that the transformed adjoint system of the modified LGR collocation method satisfies a discrete representation of the Weierstrass-Erdmann condition given in Eq.~\eqref{eqn:WECond}.  First, the transformations given in Eqs.~\eqref{eqn:modLGRSubX}--\eqref{eqn:modLGRSubVnp1} together with the definition of the Hamiltonian given in Eq.~\eqref{eqn:modHam}, Eq.~\eqref{eqn:modLGRKKTa} simplifies to
\begin{align} \label{eqn:lamalph}
\beta = \frac{t_f-t_0}{2}\sum^{N}_{i=1}w_i H_i\rk , \quad \left(k = 1,\hdots, K\right),
\end{align}
where $\beta$ is the Lagrange multiplier defined in Eq.~\eqref{eqn:ModLGRAugCost} associated with the constraint given in Eq.~\eqref{eqn:alphacon2}.  Next, multiplying Eq.~\eqref{eqn:lamalph} by $\alpha_k$ gives
\begin{align}\label{eqn:lamalph2}
\alpha_k\beta = \frac{t_f-t_0}{2}\alpha_k\sum^{N}_{i = 1} w_i H_i\rk,\quad \left(k = 1,\hdots, K\right).
\end{align}
Then, because $H$ is not an explicit function of time it follows that $H_i\rk$ is constant in each mesh interval.  Moreover, the right-hand side of Eq.~\eqref{eqn:lamalph2} is LGR quadrature approximation of the integral of the Hamiltonian over the interval $[T_{k-1}, T_k]$.  Consequently, using the definition of $\alpha_k$ from Eq.~\eqref{eqn:dsdtau}, Eq.~\eqref{eqn:lamalph2} can be rewritten as
\begin{equation}\label{eqn:lamalph3}
  \frac{2\alpha_k}{t_f-t_0}\beta = 2\alpha_k H\rk,\quad \left(k = 1,\hdots, K\right),
\end{equation}
Equation ~\eqref{eqn:lamalph3} then reduces to
\begin{equation}\label{eqn:HamB}
\frac{\beta}{t_f - t_0} = H\rk, \quad \left(k = 1,\hdots, K\right).
\end{equation}
The implication of Eq.~\eqref{eqn:HamB} is that the Hamiltonian must be the same value in each mesh interval.  Therefore,  Eq.~\eqref{eqn:lamalph3} can only be satisfied if the Hamiltonian is constant on the time interval $[-1,+1]$.  The transformed adjoint system of the standard LGR collocation method adjoint mapping scheme requires only that the Hamiltonian is constant within a mesh interval, but does not require that the Hamiltonian be constant across the entire time interval.  On the other hand, the modified LGR collocation mesh ensures that the Hamiltonian is constant across the entire time interval.  The following section provides an example that demonstrates the accuracy of the costate estimation method developed in Section \ref{sect:modadjoint} and compares the results of the modified LGR collocation method with the results obtained using the standard LGR collocation method.  

\subsection{Example of Costate Estimate\label{sect:exampleCostate}}

In this section the costate estimate arising from the modified LGR collocation method is demonstrated on the example problem given in Eqs.~\eqref{eqn:example} of Section~\ref{sect:example}.  For comparison, the exact switch point was hard coded into the standard LGR method.   The dual variables returned by the NLP solver are shown in Fig.~\ref{fig:dualvarcomp1} and \ref{fig:dualvarcomp2}.  Figure \ref{fig:dualvarcomp1} shows that the dual variables returned for the $\dot{x}$ approximation are exactly the same.  Figure \ref{fig:dualvarcomp2} shows a difference in the dual variables associated with the approximation $\dot{V}$ of $\dot{v}$, with the two dual variables of the modified LGR method located at the switch time ($\tau = 0$) arising from the additional collocation conditions associated with those differential equations that are a function of the control.
\begin{figure}[ht!]
\centering
\subfloat[Dual variable, $\Lambda_x$, for two-interval formulation of example given in Eq.~\eqref{eqn:exampleMI} using both the standard and modified LGR collocation.\label{fig:dualvarcomp1}]{\includegraphics[scale=0.45]{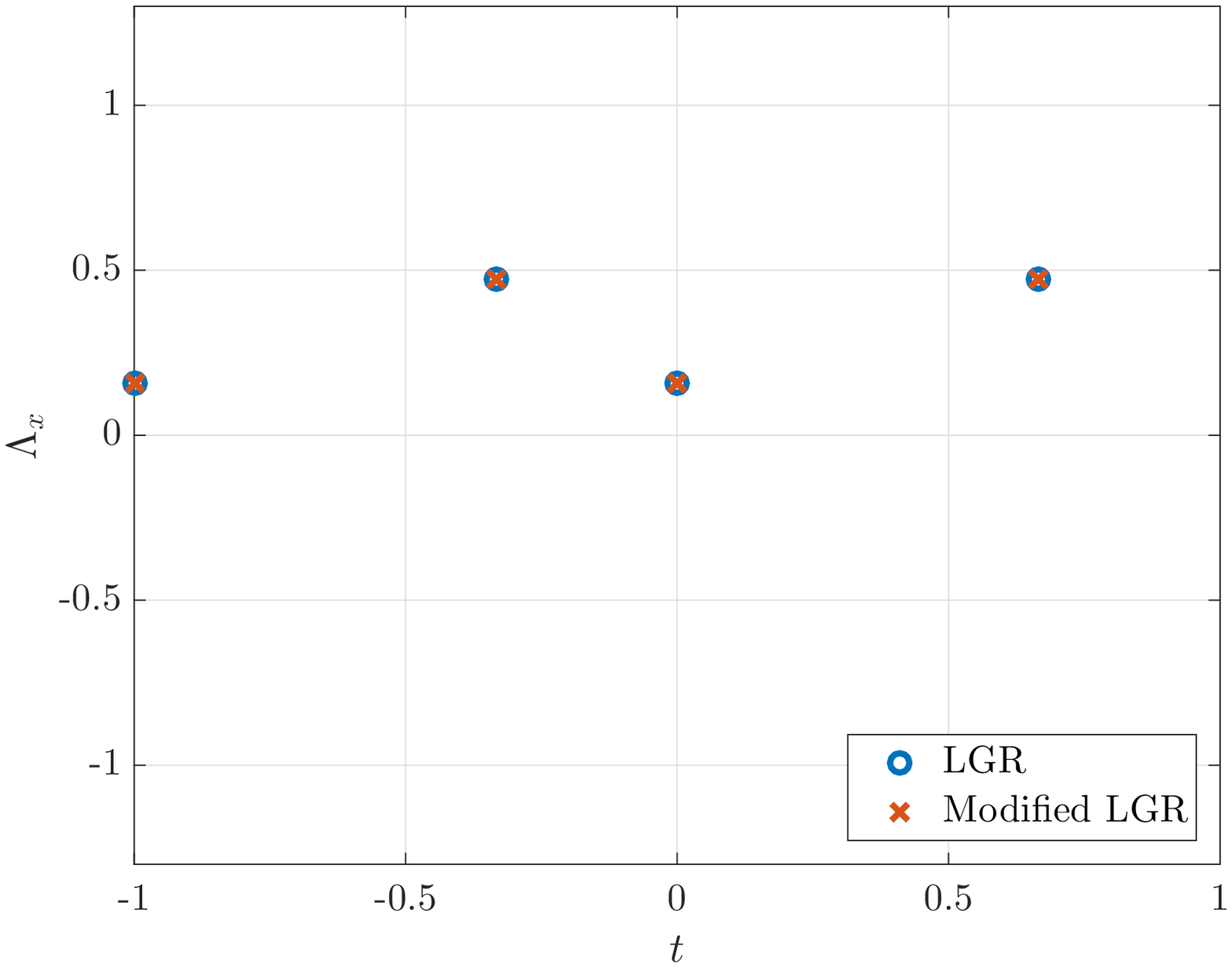}}~~\subfloat[Dual variable, ${\Lambda}_v$, for two-interval formulation of example given in Eq.~\eqref{eqn:exampleMI} using both the standard and modified LGR collocation.\label{fig:dualvarcomp2}]{\includegraphics[scale=0.45]{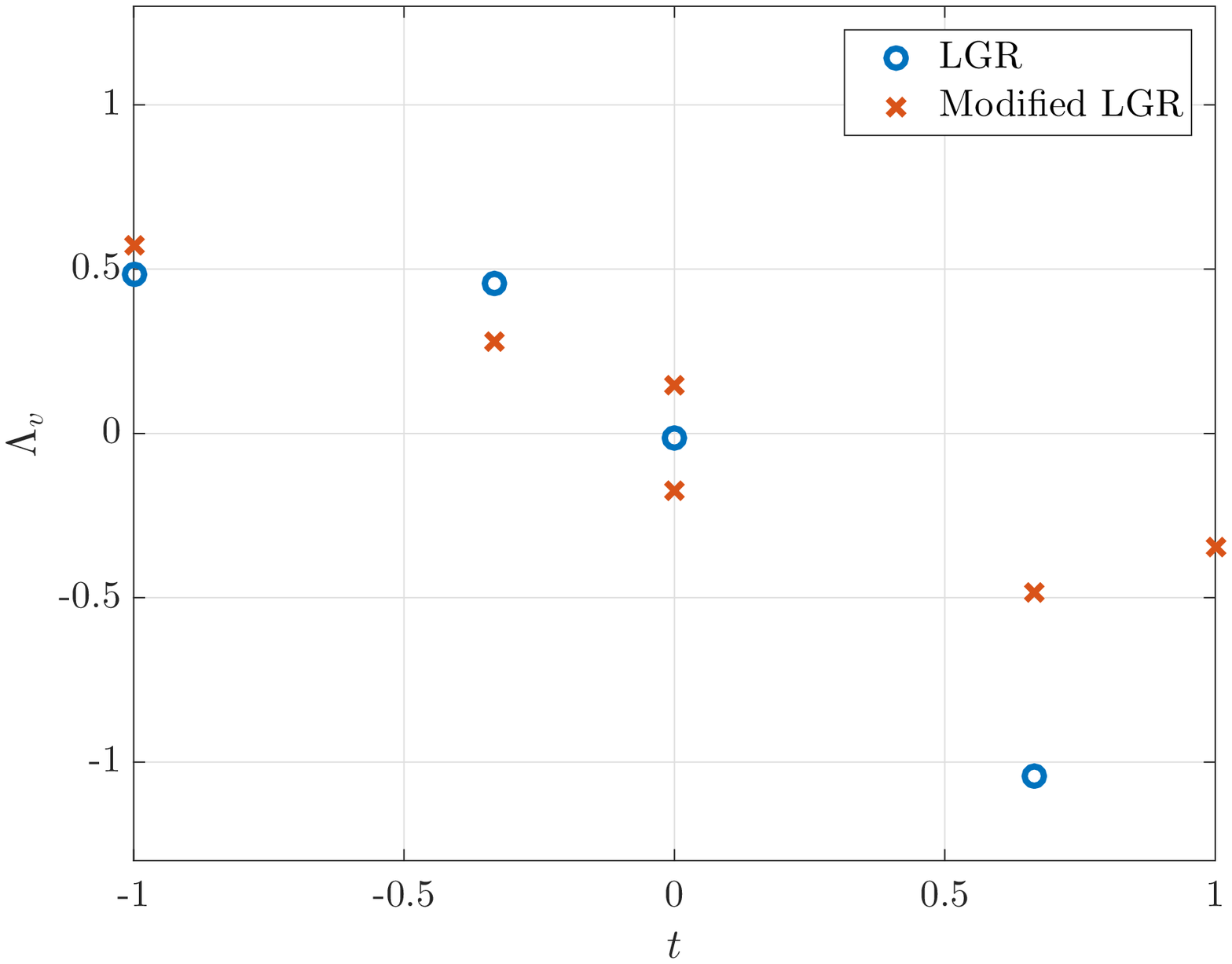}}
\caption{Dual variables $\Lambda_x$ and $\Lambda_v$ for the example problem using both the standard and modified LGR collocation methods.}
\end{figure}

\begin{figure}
\subfloat[Costate estimate, $\lambda_x(t)$, for two-interval formulation of example given in Eq.~\eqref{eqn:exampleMI} using both the standard and modified LGR collocation.\label{fig:costatevarcomp1}]{\includegraphics[scale=0.45]{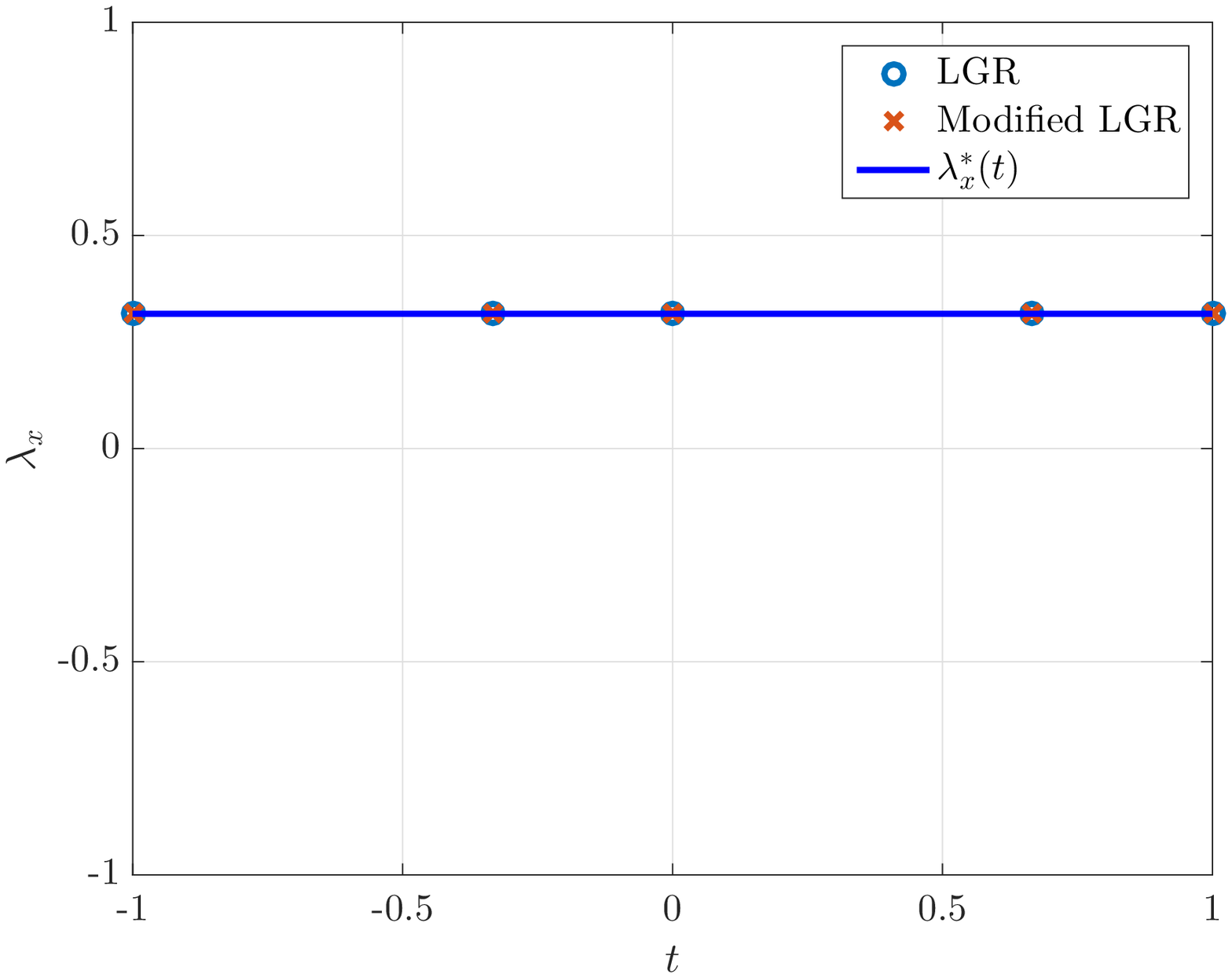}}~~\subfloat[Costate estimate, ${\lambda}_v(t)$, for two-interval formulation of example given in Eq.~\eqref{eqn:exampleMI} using both the standard and modified LGR collocation.\label{fig:costatevarcomp2}]{\includegraphics[scale=0.45]{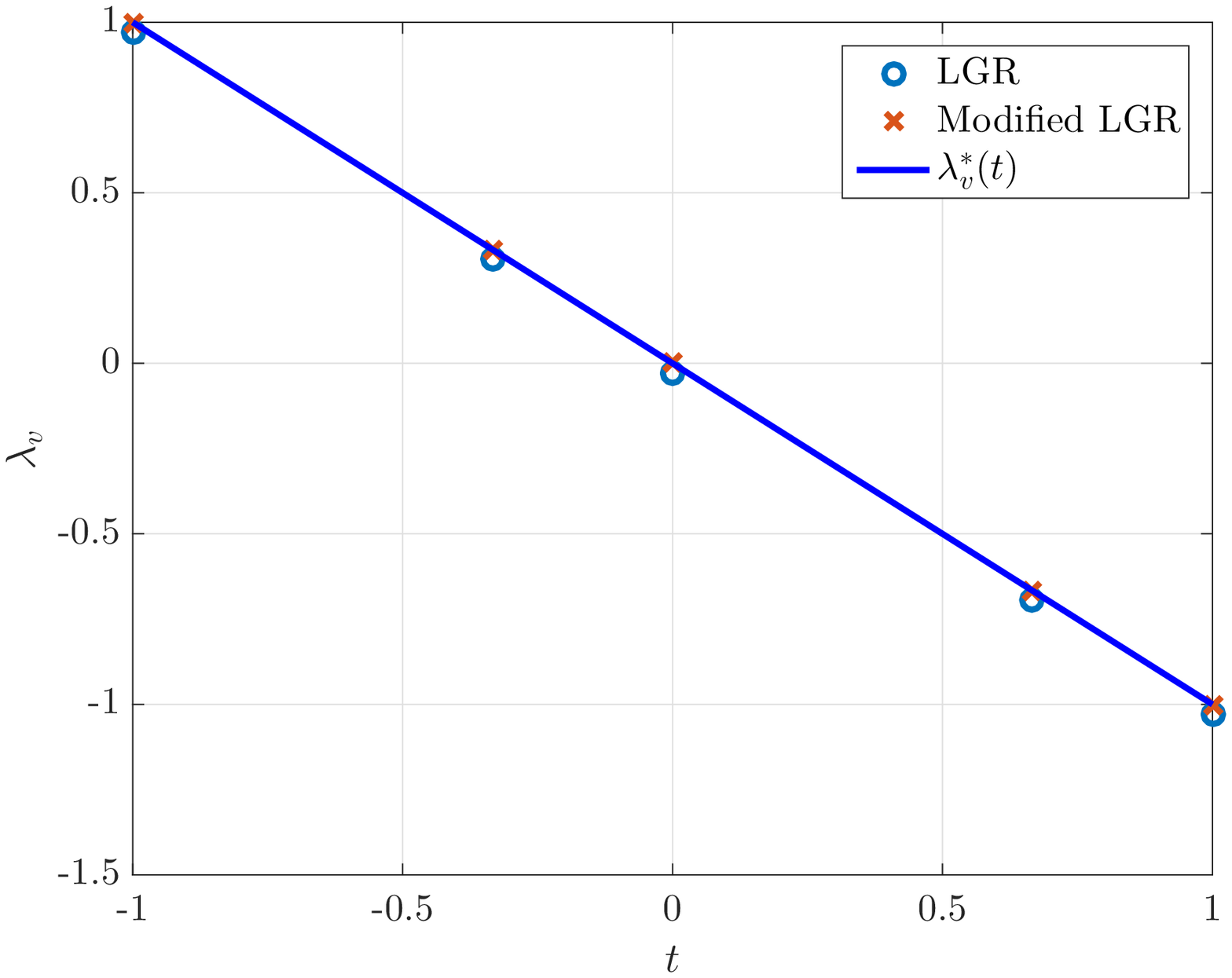}}
\caption{Costate estimates $\lambda_x$ and $\lambda_v$ for the example problem using both the standard and modified LGR collocation methods.}
\end{figure}

Figures \ref{fig:costatevarcomp1} and \ref{fig:costatevarcomp2} shows the costate approximations obtained using the standard LGR method and the modified LGR method.  Both methods return the correct value for $\lambda(t)$.  Note, however, that the estimate for ${\lambda}_v(t)$ is not correct when the standard LGR method is implemented with the switch time fixed at its exact value.  The fact that the approximation of ${\lambda}_v(t)$ is incorrect when using the exact switch time in the standard LGR method implies that the location of the switch time computed by the standard LGR method will also be incorrect.  
\begin{figure}[ht!]
\centering
\subfloat[Hamiltonian, $\mathcal{H}$, for the standard LGR method when the correct switch time, $T_1^* = 0$, is provided to the NLP solver.\label{fig:Ham}]{\includegraphics[scale=0.45]{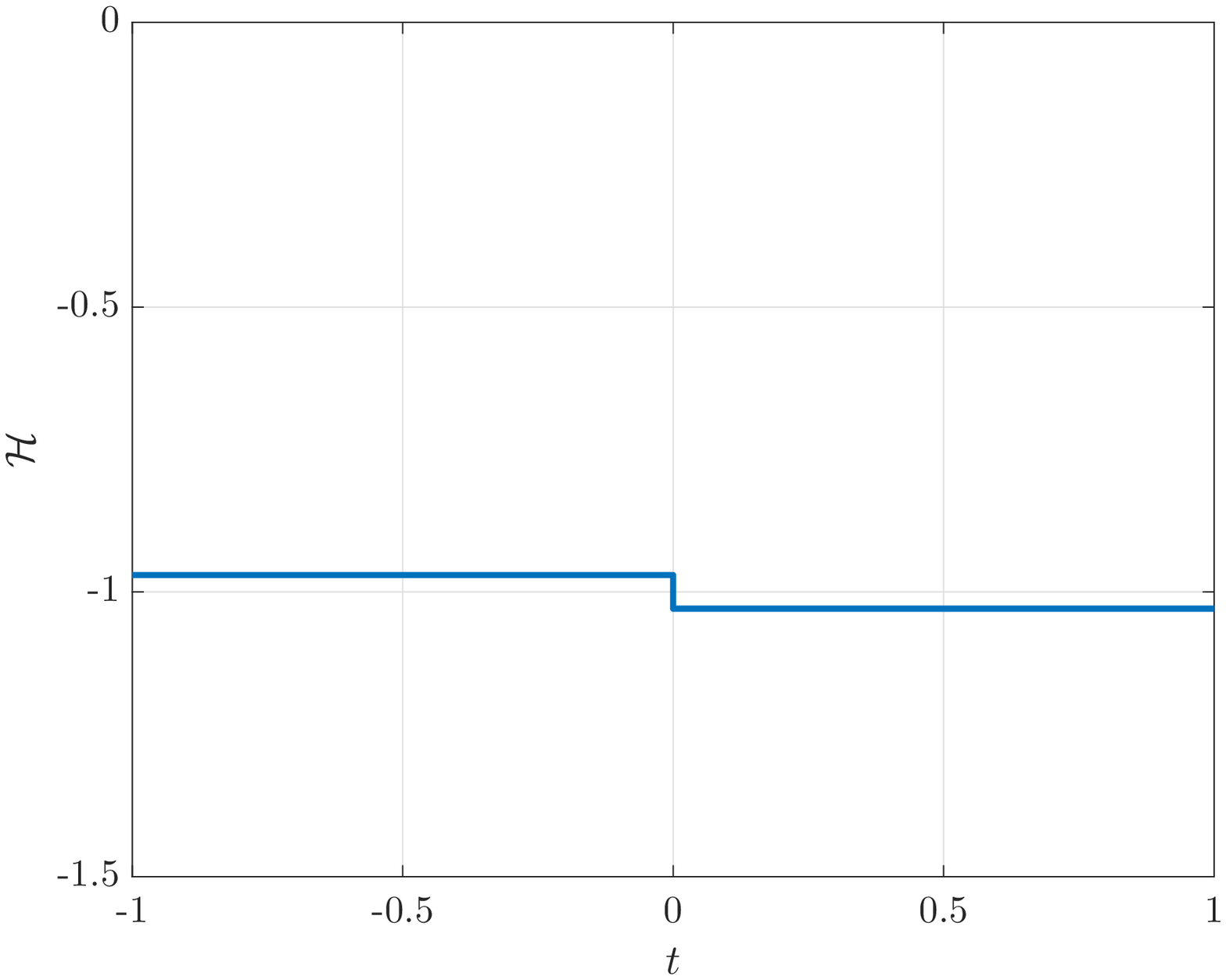}}~~\subfloat[Hamiltonian, $\mathcal{H}$, for the modified LGR method when the switch time, $T_1^* = 0$, is determined by the NLP solver.\label{fig:ModHam}]{\includegraphics[scale=0.45]{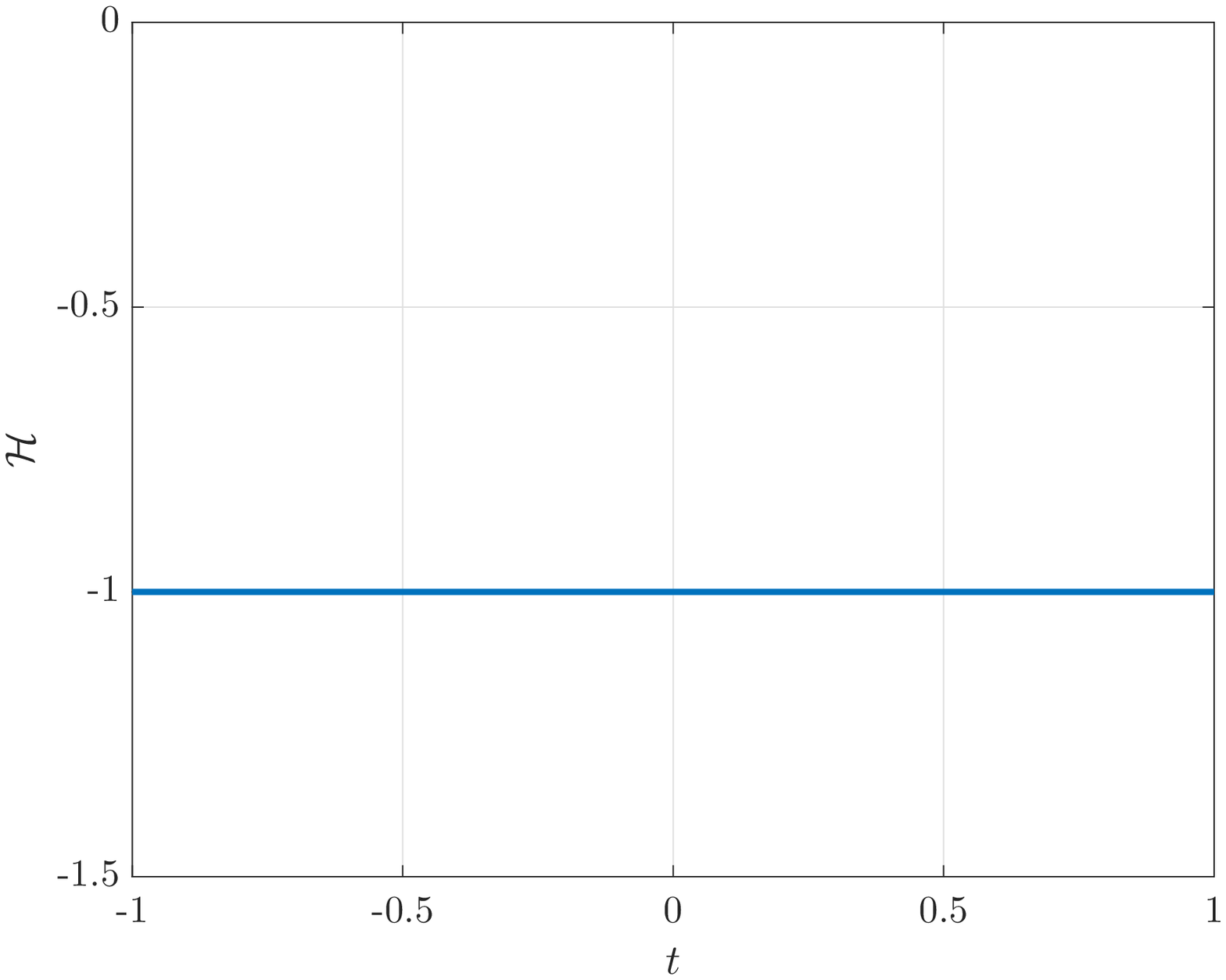}}
\caption{Hamiltonian for both the standard and modified LGR collocation method.}
\end{figure}

Figures \ref{fig:Ham} and \ref{fig:ModHam} demonstrate further the problem when using the standard LGR method when the switch time fixed at its exact value.  While the integral from $-1$ to $+1$ in Fig. \ref{fig:Ham} is correct and each interval has a continuous and constant Hamiltonian,  the integral from $-1$ to $T_1$ and from $T_1$ to $+1$ is incorrect.  The clear discontinuity in the Hamiltonian from Fig.~\ref{fig:Ham}  shows that the Weierstrass-Erdmann conditions from Eq.~\eqref{eqn:WECond} are not satisfied by the standard LGR method.  Furthermore, the discontinuity in Fig.~\ref{fig:Ham} is a result of the incorrect costate that is returned from the standard LGR method as seen in Fig.~\ref{fig:costatevarcomp2}.  Figure \ref{fig:costatevarcomp2} shows that the ${\lambda}_v(T_1)\neq 0$ for the standard LGR method, so not only are the Weierstrass-Erdmann conditions not satisfied, but neither are the standard necessary conditions for optimality.  Figure \ref{fig:ModHam} demonstrates that the additional constraint from Eq.~\eqref{eqn:HamB} enforces continuity throughout the Hamiltonian thus satisfying the Weierstrass-Erdmann conditions of Eq.~\eqref{eqn:WECond}.

\clearpage

\section{Conclusions\label{sect:conclusions}}

A new method has been developed for solving optimal control problems whose solutions are nonsmooth.  The standard LGR collocation method has been modified to include two variables and two constraints at the end of a mesh interval.  These new variables are the time associated with the intersection of mesh intervals and the value of the control at the end of the each mesh interval.  The two additional constraints are a collocation condition on each differential equation that is a function of control and an inequality constraint on the control at the endpoint of each mesh interval.  These additional constraints modify the search space of the nonlinear programming problem such that an accurate approximation to the location of the nonsmoothness is obtained.  A transformation of the Lagrange multipliers of the NLP to the costate of the optimal control problem has then been developed and the resulting transformed adjoint system of the modified Legendre-Gauss-Radau method has then been derived.  Furthermore, it has been shown that the costate estimate satisfies the Weierstrass-Erdmann optimality conditions.  Finally, an example is used throughout the paper to motivate the various aspects of the discussion.  

\section*{Acknowledgments}

The authors gratefully acknowledge support for this research from the U.S.~Office of Naval Research under grants N00014-15-1-2048 and N00014-19-1-2543, from the U.S.~National Science Foundation under grants CBET-1404767, DMS-1522629, DMS-1819002, DMS-1924762, and CMMI-1563225.  

\renewcommand{\baselinestretch}{1}
\normalsize
\normalfont
\bibliographystyle{aiaa}
% \bibliography{References}

\end{document}